\documentclass{amsart}
\usepackage{amsmath,amsthm,amscd,amssymb,amsfonts, amsbsy}
\usepackage{latexsym}
\usepackage{txfonts, mathrsfs}
\usepackage{color}
\usepackage{enumerate}

\numberwithin{equation}{section}

\theoremstyle{plain}
\newtheorem{theorem}[equation]{Theorem}
\newtheorem{lemma}[equation]{Lemma}
\newtheorem{corollary}[equation]{Corollary}

\theoremstyle{definition}

\newtheorem{example}[equation]{Example}

\newtheorem*{CLB}{Condition (LB)}
\newtheorem*{CIH}{Condition (IH)}
\newtheorem*{CLH}{Condition (LH)}
\newtheorem*{CEM}{Condition (S)}

\theoremstyle{remark}
\newtheorem{remark}[equation]{Remark}

\newtheorem*{case1}{Case 1: $\abs{x-y} \leq \sqrt{t-s}<R_{max}$}
\newtheorem*{case2}{Case 2: $\sqrt{t-s}< \abs{x-y} \wedge R_{max}$}
\newtheorem*{case3}{Case 3: $R_{max} \leq \sqrt{t-s}$}
\newtheorem*{casei}{Case 1: $\abs{x-y} \leq d_x\wedge d_y$}
\newtheorem*{caseii}{Case 2: $d_x\wedge d_y < \abs{x-y} <d_x\vee d_y$}
\newtheorem*{caseiii}{Case 3: $d_x\vee d_y \leq \abs{x-y}$}
\newtheorem*{parta}{Part (a)}
\newtheorem*{partb}{Part (b)}

\newcommand{\supp}{\operatorname{supp}}
\newcommand{\dist}{\operatorname{dist}}
\newcommand{\diam}{\operatorname{diam}}

\newcommand{\osc}{\operatorname*{osc}}
\newcommand{\esssup}{\operatorname*{ess\,sup}}

\newcommand{\bR}{\mathbb R}

\newcommand\cA{\mathcal{A}}

\newcommand\cG{\mathcal{G}}

\newcommand\cL{\mathscr{L}}

\newcommand\cS{\mathcal{S}}
\newcommand\cU{\mathcal{U}}
\newcommand\cV{\mathscr{V}}
\newcommand\cW{\mathscr{W}}

\newcommand\rV{\mathring{\cV}}

\newcommand\cLt{{}^t\!\mathscr{L}}
\newcommand\Lt{{}^t\!L}

\newcommand{\ip}[1]{\left\langle#1\right\rangle}

\providecommand{\set}[1]{\{#1\}}
\providecommand{\Set}[1]{\left\{#1\right\}}
\providecommand{\bigset}[1]{\bigl\{#1\bigr\}}

\providecommand{\abs}[1]{\lvert#1\rvert}

\providecommand{\bigabs}[1]{\bigl\lvert#1\bigr\rvert}
\providecommand{\Bigabs}[1]{\Bigl\lvert#1\Bigr\rvert}

\providecommand{\norm}[1]{\lVert#1\rVert}

\providecommand{\bignorm}[1]{\bigl\lVert#1\bigr\rVert}

\providecommand{\tri}[1]{\lvert\lVert#1\rvert\rVert}

\renewcommand{\epsilon}{\varepsilon}
\renewcommand{\vec}[1]{\boldsymbol{#1}}
\renewcommand{\qedsymbol}{$\blacksquare$}

\newcommand{\LB}{\mathrm{(LB)}}
\newcommand{\IH}{\mathrm{(IH)}}
\newcommand{\PH}{\mathrm{(PH)}}
\newcommand{\LH}{\mathrm{(LH)}}
\newcommand{\LHP}{\mathrm{(LH{}^\prime)}}
\newcommand{\CS}{\mathrm{(S)}}
\newcommand{\VMO}{\mathrm{VMO}}

\begin{document}
\title[Green's matrix for parabolic systems]{Global estimates for Green's matrix of second order parabolic systems with application to elliptic systems in two dimensional domains}
\author[S. Cho ]{Sungwon Cho }
\address[S. Cho]{Department of Mathematics, Yonsei University, Seoul 120-749, Republic of Korea}
\email{harnack@yonsei.ac.kr}

\author[H. Dong]{Hongjie Dong}
\address[H. Dong]{Division of Applied Mathematics, Brown University, 182 George Street, Box F, Providence, RI 02912, USA}
\email{Hongjie\_Dong@brown.edu}

\author[S. Kim]{Seick Kim}
\address[S. Kim]{Department of Mathematics, Yonsei University, Seoul 120-749, Republic of Korea}
\curraddr{Department of Computational Science and Engineering, Yonsei University, Seoul 120-749, Republic of Korea}
\email{kimseick@yonsei.ac.kr}


\begin{abstract}
We establish global Gaussian estimates for the Green's matrix of divergence form, second order parabolic systems in a cylindrical domain under the assumption that weak solutions of the system vanishing on a portion of the boundary satisfy a certain local boundedness estimate and a local H\"older estimate.
From these estimates, we also derive global estimates for the Green's matrix for elliptic systems with bounded measurable coefficients in two dimensional domains.
We present a unified approach valid for both the scalar and vectorial cases and discuss several applications of our result.
\end{abstract}
\maketitle

\section{Introduction}

Fundamental solutions of parabolic equations in divergence form with bounded measurable coefficients have been a subject of research for many years.
The first significant step in this direction was made in 1958 by Nash \cite{Nash}, who established certain estimates of the fundamental solutions in proving local H\"older continuity of weak solutions.
In 1967, Aronson \cite{Aronson} proved two-sided Gaussian estimates for the fundamental solutions by using the parabolic Harnack inequality of Moser \cite{Moser}.
In 1986, Fabes and Stroock \cite{FS} showed that the idea of Nash could be used to establish Aronson's Gaussian bounds, which consequently gave a new proof of Moser's parabolic Harnack inequality.
There are many books and articles related to this subject;  see e.g., \cite{Davies89, PE, Robinson, VSC} and references therein.

Compared to a long history in the study of fundamental solutions for parabolic equations with real coefficients, there has been relatively little study on the fundamental matrices for parabolic systems until recently, except when the coefficients are sufficiently regular; see  e.g., Eidel'man \cite{Eidelman}.
In 1996, Auscher \cite{Auscher} gave a new proof of Aronson's Gaussian upper bound for the fundamental solution of parabolic equations with time independent coefficients, which carries over to the case of $L^\infty$-complex perturbations of real coefficients.
We note that a parabolic equation with complex coefficients is, in fact, a special case of a system of parabolic equations.
Since then, there has been active research in this direction; see e.g., \cite{AMcT, AQ, AT, AT2, CDK, HK04, Kim} for related results.

In this article, we study the Green's matrix for parabolic systems
\begin{equation}    \label{eq0.0}
\partial u^i/\partial t -\sum_{j=1}^m \sum_{\alpha,\beta=1}^ n D_\alpha \left(A^{\alpha\beta}_{ij}(t,x) D_\beta\, u^j\right),\quad D_\alpha = \partial/\partial {x_\alpha},\quad  i=1,\ldots,m,
\end{equation}
in a cylindrical domain $\cU=\bR \times \Omega$, where $\Omega $ is a (possibly unbounded) domain in $\bR^n$.
We assume that the coefficients are measurable functions defined in the whole space $\bR^{n+1}$ satisfying the strong parabolicity condition
\begin{equation}    \label{eqP-02}
\sum_{i,j=1}^m \sum_{\alpha,\beta=1}^n A^{\alpha\beta}_{ij}(t,x)\xi^j_\beta \xi^i_\alpha \geq \nu\sum_{i=1}^m \sum_{\alpha=1}^n \bigabs{\xi^i_\alpha}^2=:\nu \bigabs{\vec \xi}^2, \quad\forall \vec\xi \in \bR^{nm},\quad\forall (t,x)\in\bR^{n+1},
\end{equation}
and also the uniform boundedness condition
\begin{equation}    \label{eqP-03}
\sum_{i,j=1}^m \sum_{\alpha,\beta=1}^n \Bigabs{A^{\alpha\beta}_{ij}(t,x)}^2\le \nu^{-2},\quad\forall (t,x)\in\bR^{n+1},
\end{equation}
for some constant $\nu\in (0,1]$.
We point out that the coefficients are assumed to be neither time independent nor symmetric.
We will later impose some further assumptions on the operator but not explicitly on its coefficients.

We are interested in the following global Gaussian estimate for the Green's matrix $\vec \cG(t,x,s,y)$ of the parabolic system \eqref{eq0.0} in a cylindrical domain $\cU=\bR\times\Omega$:
For any $T>0$, there exists a positive constant $C$ such that for all $t,s\in \bR$ satisfying $0<t-s<T$ and $x,y \in \Omega$, we have
\begin{equation}        \label{eq0.1}
\abs{\vec \cG (t,x,s,y)} \leq \frac{C}{(t-s)^{n/2}}\exp\Set{-\frac{\kappa\abs{x-y}^2}{t-s}},
\end{equation}
where $\kappa$ is a positive constant independent of $T$.
We are also interested in the following global estimate for the Green's matrix $\vec \cG(t,x,s,y)$ of the parabolic system \eqref{eq0.0} in  $\cU=\bR\times\Omega$ when the base $\Omega$ is subject to a certain condition: For any $T>0$, there exists a positive constant $C$ such that for all $t,s\in \bR$ satisfying $0<t-s<T$ and $x,y \in \Omega$, we have
\begin{multline}							\label{eq0.2}
\abs{\vec \cG(t,x,s,y)} \leq C \left(1 \wedge \frac {d_x} {\sqrt {t-s} \vee \abs{x-y}} \right)^{\mu} \left(1 \wedge \frac{d_y} {\sqrt {t-s} \vee \abs{x-y}}\right)^{\mu}\\
\cdot (t-s)^{-n/2}\exp \set{-\kappa \abs{x-y}^2/(t-s)},
\end{multline}
where $\kappa>0$ and $\mu\in (0,1]$ are constants independent of $T$, and we used the notation $a\wedge b=\min(a,b)$, $a\vee b=\max(a,b)$, and $d_x=\dist(x,\partial\Omega)$.

The goal of this article is to present how one can derive the global estimate like \eqref{eq0.1} for the Green's matrix
using a local boundedness estimate for the weak solutions of the parabolic system vanishing on a portion of the boundary; see Condition $\LB$ below for the precise statement of the local boundedness estimate.
In fact, we show that such a local boundedness estimate is a necessary and sufficient condition for the Green's matrix of the system to have a global estimate like \eqref{eq0.1}.
We shall also show how to derive a global estimate like \eqref{eq0.2} for the Green's matrix
by using a local H\"older continuity estimate for the solution of the system vanishing on a portion of the boundary; see Condition $\LH$ below for the statement of the local H\"older estimate.
There is a standard method in constructing the Green's matrix for elliptic systems in a domain $\Omega$ out of the ``Dirichlet heat kernel'' of the elliptic system, namely by integrating it with respect to $t$-variable from zero to infinity.
By utilizing the above global estimates \eqref{eq0.1} and \eqref{eq0.2}, we obtain the following global estimate for Green's matrix $\vec G(x,y)$ for elliptic systems with bounded measurable coefficients in a two dimensional domain $\Omega$:
\begin{equation}								\label{eq0.3}
\abs{\vec G(x,y)} \leq C\left(1\wedge \frac{d_x}{\abs{x-y}}\right)^{\mu}\left(1\wedge \frac{d_y}{\abs{x-y}}\right)^{\mu} \left\{ 1+ \ln_+ \left(\frac{1}{\abs{x-y}}\right)\right\},
\end{equation}
where $C>0$ and $\mu\in(0,1]$ are constants depending on $\Omega$, and we used the notation $\ln_+t=\max(\ln t, 0)$.
We do not consider Green's matrix for elliptic systems in a three or higher dimensional domain in this article.
For treatment of such cases, we refer to a very recent article \cite{KK09}, where conditions similar to ours were introduced.
In fact, our conditions $\LB$ and $\LH$ are motivated by the corresponding elliptic conditions appeared in \cite{KK09}.
We point out that most of the results in \cite{KK09} can be also obtained by following the above mentioned ``Dirichlet heat kernel'' approach and utilizing the global estimates like \eqref{eq0.1} and \eqref{eq0.2}, albeit it would be far more complicated to do so.
The axiomatic approach adopted in this article is also in the spirit of \cite{HK04, HK07}.

The novelty of our work is in presenting a unifying method that establishes the global estimates \eqref{eq0.1} and \eqref{eq0.2} for the Green's function for the uniformly parabolic equations with bounded measurable coefficients as well as for the Green's matrix of the parabolic systems \eqref{eq0.0}, for instance, with coefficients uniformly continuous or VMO in $x$-variables and bounded measurable in $t$-variable, in a cylindrical domain $\bR\times\Omega$, where the base $\Omega$ is a $C^1$ domain or a Lipschitz domain with a sufficiently small Lipschitz constant.
Also, this article provides a unified approach in establishing the global estimate \eqref{eq0.3} for both Green's function for elliptic equations and Green's matrix for elliptic systems with bounded measurable coefficients in two dimensional domains.
In a recent article \cite{CDK}, we proved the existence of Green's matrix for the parabolic system \eqref{eq0.0} in an arbitrary cylindrical domain $\cU=\bR\times\Omega$ under the assumption that weak solutions of the systems satisfy an interior H\"older continuity estimate.
We also derived various local estimates for the Green's matrix under the same assumption but, however, the global Gaussian estimate like \eqref{eq0.1} was proved only in the case when $\Omega=\bR^n$.
In another recent article \cite{DK09}, the second and third named authors proved the existence and local estimates of Green's matrix for elliptic systems with bounded measurable coefficients in a two dimensional domain $\Omega$ with finite area or width, by utilizing the estimates established in \cite{CDK}.
However, again, the global estimate similar to \eqref{eq0.3} was established only when $\Omega$ is the open set above a Lipschitz graph $\varphi:\bR\to\bR$.
In this sense, the present article may be considered as a sequel  of both articles \cite{CDK} and \cite{DK09} with a considerable improvement.

The organization of the paper is as follows. In Section~\ref{sec:nd}, we introduce some notation and definitions including our definition of the Green's matrix of the parabolic system \eqref{eq0.0} in $\cU=\bR\times\Omega$.
In Section~\ref{main}, we give precise statement of the conditions $\LB$ and $\LH$ and state our main theorems.  In Section~\ref{sec:app}, we present some applications of our main results.
The proofs of our main results are given in Section~\ref{sec:p} and several technical lemmas are proved in Appendix.
Section~\ref{sec:2d} is devoted to the discussion of global estimates for Green's matrices for elliptic systems with bounded measurable coefficients in two dimensional domains and Section~\ref{sec:h} is allocated to a brief discussion regarding the global estimates for systems with H\"older or Dini continuous coefficients.

Finally, a few remarks are in order.
Green's functions for uniformly elliptic equations with bounded measurable coefficients were extensively studied in classical papers \cite{GW, LSW}.
Green's matrices for elliptic systems were investigated earlier in \cite{DM, Fuchs84, Fuchs86}.
In \cite{AT},  Auscher and Tchamitchian introduced the ``Dirichlet property (D)'' in connection with the Gaussian estimates for the heat kernel of the elliptic operator with complex coefficients, which is related to the condition $\LH$ of this article; see Remark~\ref{rmk:P-03}.

\section{Notation and Definitions}					\label{sec:nd}

We mainly follow the notation and definitions of \cite{CDK}, most of which were in turn chosen to be compatible with those used in \cite{LSU}.
\subsection{Basic notation}

Let $\cL$ be a parabolic operator acting on column vector valued functions $\vec u=(u^1,\ldots,u^m)^T$ defined on a domain in $\bR^{n+1}$ in the following way:
\begin{equation}						\label{eq2.01aa}
\cL\vec u = \vec u_t-D_\alpha \bigl(\vec A^{\alpha\beta}\, D_\beta \vec u\bigr),
\end{equation}
where we use the usual summation convention over repeated indices $\alpha,\beta=1,\ldots, n$, and $\vec A^{\alpha\beta}=\vec A^{\alpha\beta}(t,x)$ are $m\times m$ matrix valued functions on $\bR^{n+1}$ with entries $A^{\alpha\beta}_{ij}$ that satisfy the conditions \eqref{eqP-02} and \eqref{eqP-03}.
Notice that the $i$-th component of the column vector $\cL \vec u$ is presented in \eqref{eq0.0}.
The adjoint operator $\cLt$ is defined by
\[\cLt \vec u = -\vec u_t-D_\alpha \bigl({}^t\!\vec A^{\alpha\beta} D_\beta \vec u\bigr),\]
where ${}^t\!\vec A^{\alpha\beta}=\bigl(\vec A^{\beta\alpha}\bigr)^T$; i.e., ${}^t\!A^{\alpha\beta}_{ij}=A^{\beta\alpha}_{ji}$.
Notice that the coefficients ${}^t\!A^{\alpha\beta}_{ij}$ satisfy the conditions \eqref{eqP-02} and \eqref{eqP-03} with the same constant $\nu$.

We use $X=(t,x)$ to denote a point in $\bR^{n+1}$; $x=(x_1,\ldots, x_n)$ will always be a point in $\bR^n$.
We also write $Y=(s,y)$, $X_0=(t_0,x_0)$, etc.
We define the parabolic distance between the points $X=(t,x)$ and $Y=(s,y)$ in $\bR^{n+1}$ as
\[\abs{X-Y}_p:=\max(\sqrt{\abs{t-s}},\abs{x-y}),\]
where $\abs{\,\cdot\,}$ denotes the usual Euclidean norm.
For an open set $\cU\subset\bR^{n+1}$, we denote
\[
d_X=\dist(X,\partial_p\cU)=\inf\bigset{\abs{X-Y}_p: Y\in \partial_p\cU};\quad \inf \emptyset = \infty,
\]
where $\partial_p\cU$ denotes the usual parabolic boundary of $\cU$.

For a given function $u=u(X)=u(t,x)$ defined on $Q \subset \bR^{n+1}$, we use
$D_i u$ for $\partial u/\partial x_i$  while we use $u_t$ (or sometimes $D_t u$) for $\partial u/\partial t$.
We also write $Du$ (or sometimes $D_x u$) for the vector $(D_1 u,\ldots,D_n u)$.
For $\mu \in (0,1]$, we define
\[
\abs{u}_{\mu/2,\mu;Q}=\abs{u}_{0;Q}+[u]_{\mu/2,\mu;Q}: = \sup_{X\in Q}\,\abs{u(X)}+\sup_{\substack{X, Y \in Q\\ X\neq Y}} \frac{\abs{u(X)-u(Y)}}{\abs{X-Y}_p^\mu}.
\]
By $\mathscr{C}^{\mu/2,\mu}(Q)$ we denote the set of all bounded measurable functions $u$ on $Q$ for which $\abs{u}_{\mu/2,\mu;Q}$ is finite.
We use the following notation for basic cylinders in $\bR^{n+1}$:
\begin{align*}
Q^-_r(X)&=(t-r^2,t)\times B_r(x);\\
Q^+_r(X)&=(t,t+r^2)\times B_r(x);\\
Q_r(X)&=(t-r^2,t+r^2)\times B_r(x),
\end{align*}
where $B_r(x)$ is the usual Euclidean ball of radius $r$ centered at $x\in \bR^n$.
For an open set $\cU\subset\bR^{n+1}$, we define
\[
\cU_r (X) =  \cU\cap Q_r (X),\quad \cS_r(X) =\partial_p \cU \cap Q_r(X),
\]
and similarly $\cU_r^\pm(X)$ and $\cS_r^\pm(X)$.
We write $\cU(t_0)$ for the set of all points  $(t_0,x)$ in $\cU$ and $I(\cU)$  for the set of all $t$ such that $\cU(t)$ is nonempty.
We denote
\[
\tri{u}_{\cU}^2= \norm{Du}_{L^2(\cU)}^2+\esssup_{t\in I(\cU)}\, \norm{u(t,\cdot)}_{L^2(\cU(t))}^2.
\]
Finally, we denote $a\wedge b= \min(a,b)$ and $a\vee b= \max(a,b)$.

\subsection{Function spaces}		\label{sec2.2}
Throughout this section, we shall always denote by $Q$ the cylinder $(a,b)\times\Omega$, where $-\infty<a<b<\infty$ and $\Omega$ is an open connected set in $\bR^n$.
We denote by $\cW^{0,1}_2(Q)$ the Hilbert space with the inner product
\[
\ip{u,v}_{\cW^{0,1}_2(Q)}:=\int_{Q} uv+\sum_{k=1}^n \int_{Q} D_k u D_k v
\]
and by $\cW^{1,1}_2(Q)$ the Hilbert space with the inner product
\[
\ip{u,v}_{\cW^{1,1}_2(Q)}:=\int_{Q} uv+\sum_{k=1}^n \int_{Q} D_k u D_k v+\int_{Q} u_t v_t.
\]
We define $\cV_2(Q)$ as the Banach space consisting of all elements of $\cW^{0,1}_2(Q)$ having a finite norm $\norm{u}_{\cV_2(Q)}:= \tri{u}_{Q}$.
The space $\cV^{0,1}_2(Q)$ is obtained by completing the set $\cW^{1,1}_2(Q)$ in the norm of $\cV_2(Q)$.

Let $\cS \subset \overline Q$ and $u$ be a $\cW^{0,1}_2(Q)$ function.
We say that $u$ vanishes (or write $u=0$) on $\cS$ if $u$ is a limit in $\cW^{0,1}_2(Q)$ of a sequence of functions in $C^\infty_c(\overline Q\setminus \cS)$.
We define $\rV_2(Q)$ (resp. $\rV^{0,1}_2(Q)\,$) the set of all functions $u$ in $\cV_2(Q)$ (resp. $\cV^{0,1}_2(Q)\,$) that vanishes on the lateral boundary $\partial_x Q:=(a,b)\times \partial\Omega$ of $Q$.
By a well known Sobolev-like embedding theorem (see e.g., \cite[\S II.3]{LSU}),
we have
\begin{equation} \label{eqn:2.2}
\norm{u}_{L^{2+4/n}(Q)} \leq C(n)\, \tri{u}_{Q}, \quad\forall u\in \rV_2(Q).
\end{equation}

The space $\cW^{0,1}_q(\cU)$ ($1\leq q <\infty$) denotes the Banach space consisting of functions $u\in L^q(\cU)$ with weak derivatives $D_\alpha u \in L^q(\cU)$ ($\alpha=1,\ldots,n$) with the norm
\[
\norm{u}_{\cW^{0,1}_q(\cU)}=\norm{u}_{L^q(\cU)}+\norm{Du}_{L^q(\cU)}.
\]
We write $u \in L^\infty_c(\cU)$ if $u \in L^\infty(\cU)$ has a support in $\overline \cU$.

\subsection{Weak Solutions}
For $\vec f, \vec g_\alpha \in L^1_{loc}(\cU)^m$ ($\alpha=1,\ldots,n$), we say that $\vec u$ is a weak solution of $\cL \vec u=\vec f+ D_\alpha\vec g_\alpha$ in $\cU$ if $\vec u \in \cV_2(\cU)^m$ and satisfies
\begin{equation*}  
-\int_{\cU} u^i\phi^i_t+ \int_{\cU} A^{\alpha\beta}_{ij}D_\beta u^j D_\alpha\phi^i= \int_{\cU} f^i \phi^i- \int_{\cU} g^i_\alpha D_\alpha\phi^i, \quad\forall \vec \phi \in C^\infty_c (\cU)^m.
\end{equation*}
Similarly, we say that $\vec u$ is a weak solution of $\cLt \vec u=\vec f+D_\alpha\vec g_\alpha$ in $\cU$ if $\vec u\in \cV_2(\cU)^m$ and satisfies
\begin{equation}  \label{eqn:E-71b}
\int_{\cU} u^i\phi^i_t+ \int_{\cU}{}^t\!A^{\alpha\beta}_{ij}D_\beta u^j D_\alpha\phi^i=
\int_{\cU} f^i \phi^i- \int_{\cU} g^i_\alpha D_\alpha\phi^i, \quad\forall \vec\phi \in C^\infty_c (\cU)^m.
\end{equation}

\subsection{Green's matrix}
Let  $\cU=\bR\times\Omega$ be a cylindrical domain, where $\Omega$ is an open connected set in $\bR^n$.
We say that an $m\times m$ matrix valued function $\vec \cG(X,Y)=\vec \cG(t,x,s,y)$, with entries $\cG_{ij} (X,Y)$ defined on the set $\bigset{(X,Y)\in\cU\times\cU: X\neq Y}$, is a Green's matrix of $\cL$ in $\cU$ if it satisfies the following properties:
\begin{enumerate}[i)]
\item
$\vec \cG(\cdot,Y)\in \cW^{0,1}_{1,loc}(\cU)$ and $\cL \vec \cG(\cdot,Y) = \delta_Y I$ for all $Y\in\cU$, in the sense that
\begin{equation*}  
\int_{\cU}\left( -\cG_{ik}(\cdot,Y)\phi^i_t+ A^{\alpha\beta}_{ij} D_\beta \cG_{jk}(\cdot,Y)D_\alpha \phi^i\right) = \phi^k(Y), \quad \forall \vec \phi \in C^\infty_c(\cU)^m.
\end{equation*}
\item
$\vec \cG(\cdot,Y) \in \cV_2^{0,1}(\cU\setminus Q_r(Y))$ for all $Y\in\cU$ and $r>0$ and $\vec \cG(\cdot,Y)$ vanishes on $\partial_p\cU$.
\item
For any $\vec f=(f^1,\ldots, f^m)^T \in L^\infty_c(\cU)$, the function $\vec u$ given by
\begin{equation*}       
\vec u(X):=\int_\Omega \vec \cG(Y,X) \vec f(Y)\,dY
\end{equation*}
belongs to $\rV^{0,1}_2(\cU)$ and satisfies $\cLt \vec u=\vec f$ in the sense of \eqref{eqn:E-71b}.
\end{enumerate}
We note that part iii) of the above definition gives the uniqueness of a Green's matrix; see  \cite{CDK}.
We shall hereafter say that $\vec \cG(X,Y)$ is ``the'' Green's matrix of $\cL$ in $\cU=\bR\times\Omega$ if it satisfies all the above properties.

\section{Main results}							\label{main}

The following condition $\LB$ shall be used to obtain the global Gaussian estimates for the Green's matrix $\vec \cG(X,Y)=\vec \cG(t,x,s,y)$ of $\cL$ in $\cU=\bR\times \Omega$.
See Theorem~\ref{thm1} below.

\begin{CLB}
There exist $R_{max}\in (0,\infty]$ and $N_0>0$ so that for all $X\in\cU$ and $0<R<R_{max}$, the following holds.
\begin{enumerate}[i)]
\item
If $\vec u$ is a weak solution of $\cL \vec u=0$ in $\cU_R^-(X)$ vanishing on $\cS_R^-(X)$, then we have
\[
\norm{\vec u}_{L^\infty(\cU_{R/2}^-(X))} \le N_0 R^{-(2+n)/2} \norm{\vec u}_{L^2(\cU_R^-(X))}.
\]

\item
If $\vec u$ is a weak solution of $\cLt \vec u=0$ in $\cU_R^+(X)$ vanishing on $\cS_R^+(X)$, then we have
\[
\norm{\vec u}_{L^\infty(\cU_{R/2}^+(X))} \le N_0 R^{-(2+n)/2} \norm{\vec u}_{L^2(\cU_R^+(X))}.
\]
\end{enumerate}
\end{CLB}

The following condition $\IH$ means that weak solutions of $\cL\vec u=0$ and $\cLt \vec u=0$ enjoy interior H\"older continuity estimates with exponent $\mu_0$.
It is not hard to see that this condition is equivalent to saying that the operator $\cL$ and its adjoint $\cLt$ satisfy the property $\PH$ in \cite{CDK}; see Lemma~\ref{lem8.0a} for the proof.

\begin{CIH}
There exist $\mu_0\in (0,1]$, $R_c \in (0,\infty]$, and  $C_0>0$ so that for all $X\in \cU$ and $0<R<R_c\wedge d_X$, the following holds.
\begin{enumerate}[i)]
\item
If  $\vec u$ is a weak solution of $\cL \vec u=0$ in $Q_R^-(X)$, then we have
\[
[\vec u]_{\mu/2,\mu;Q_{R/2}^-(X)}\leq C_0 R^{-\mu_0}\left(\fint_{Q_R^-(X)} \abs{\vec u}^2\right)^{1/2}.
\]
\item
If  $\vec u$ is a weak solution of $\cLt \vec u=0$ in $Q_R^+(X)$, then we have
\[
[\vec u]_{\mu/2,\mu;Q_{R/2}^+(X)}\leq C_0 R^{-\mu_0}\left(\fint_{Q_R^+(X)} \abs{\vec u}^2\right)^{1/2}.
\]
\end{enumerate}
\end{CIH}

\begin{theorem} \label{thm1}
Let $\cU=\bR\times\Omega$ and assume the conditions $\IH$ and $\LB$.
Then the Green's matrix $\vec \cG(X,Y)$ of $\cL$ in $\cU$ exists and for all $X, Y\in \cU$ with $X \neq Y$,  we have
\begin{equation} \label{eq2.17d}
\abs{\vec \cG(t,x,s,y)}\leq C\, \chi_{(0,\infty)}(t-s)\cdot \left\{(t-s)\wedge R_{max}^2\right\}^{-n/2}\exp\left\{-\kappa \abs{x-y}^2/(t-s)\right\},
\end{equation}
where $C= C(n,m,\nu, N_0)$ and $\kappa=\kappa(\nu)$.
\end{theorem}

The following theorem says that the converse of Theorem~\ref{thm1} is also true.

\begin{theorem} 				\label{thm1c}
Assume the condition $\IH$.
Let $\vec \cG(X,Y)$ be the Green's matrix of $\cL$ in $\cU=\bR\times\Omega$.
Suppose there exist $R_{max}\in (0,\infty]$ and positive constants $C_0$ and $\kappa$ such that for all $X, Y\in \cU$ with $X \neq Y$, we have
\begin{equation}            \label{eq2.17dd}
\abs{\vec \cG(t,x,s,y)}\leq C_0\, \chi_{(0,\infty)}(t-s)\cdot \left\{(t-s)\wedge R_{max}^2\right\}^{-n/2}\exp\left\{-\kappa \abs{x-y}^2/(t-s)\right\}.
\end{equation}
Then the condition $\LB$ is satisfied with the same $R_{max}$ and $N_0=N_0(n,m,\nu, C_0, \kappa)$.
\end{theorem}

The following condition $\LH$ means, loosely speaking, that weak solutions of $\cL u=0$ and $\cLt u=0$ vanishing on a relatively open subset $\cS$ of  $\partial_p\cU$ are locally H\"older continuous up to $\cS$ with exponent $\mu_0$.

\begin{CLH}
There exist $\mu_0\in (0,1]$, $R_{max} \in (0,\infty]$, and  $N_1>0$ so that for all $X\in\cU$ and $0<R<R_{max}$, the following holds.
\begin{enumerate}[i)]
\item
If  $\vec u$ is a weak solution of $\cL \vec u=0$ in $\cU_R^-(X)$ vanishing on $\cS_R^-(X)$, then we have
\[
[\tilde{\vec u}]_{\mu/2,\mu;Q_{R/2}^-(X)}\leq N_1 R^{-\mu_0}\left(\fint_{Q_R^-(X)} \abs{\tilde{\vec u}}^2\right)^{1/2};\quad \tilde{\vec u}=\chi_{\cU_R^-(X)} \vec u.
\]
\item
If  $\vec u$ is a weak solution of $\cLt \vec u=0$ in $\cU_R^+(X)$ vanishing on $\cS_R^+(X)$, then we have
\[
[\tilde{\vec u}]_{\mu/2,\mu;Q_{R/2}^+(X)} \leq N_1 R^{-\mu_0}\left(\fint_{Q_R^+(X)} \abs{\tilde{\vec u}}^2\right)^{1/2};\quad \tilde{\vec u}=\chi_{\cU_R^+(X)} \vec u.
\]
\end{enumerate}
\end{CLH}

\begin{remark}								\label{rmk3.5}
In the above condition $\LH$, the constant $R_{max}$ is interchangeable with $c\cdot R_{max}$ for any fixed $c \in (0,\infty)$, possibly at the cost of increasing the constant $N_1$.
\end{remark}

It is not hard to see that $\LH$ implies $\LB$; see Lemma~\ref{lem2.19} in Appendix.
Also, it is obvious that $\LH$ implies $\IH$.
Therefore if  $\cU=\bR\times\Omega$ and $\LH$ is satisfied, then by Theorem~\ref{thm1}, the Green's matrix $\vec \cG(X,Y)$ of $\cL$ in $\cU$ exists and satisfies the estimate \eqref{eq2.17d}.
The following theorem says that in fact, in such a case, a better bound for the Green's matrix $\vec \cG(X,Y)$ is available near the boundary $\partial_p \cU=\bR\times\partial\Omega$.

\begin{theorem}			\label{thm2}
Let $\cU=\bR\times\Omega$ and assume the condition $\LH$.
Let $\vec \cG(X,Y)$ be the Green's matrix of $\cL$ in $\cU$ and denote
\begin{equation}			\label{eq3.6mm}
\delta(X,Y)= \left(1 \wedge \frac {d_X} {R_{max}\wedge \abs{X-Y}_p} \right) \left(1 \wedge \frac{d_Y} {R_{max}\wedge \abs{X-Y}_p}\right).
\end{equation}
Then for all $X, Y \in \cU$ with $X\neq Y$, we have
\begin{equation}			\label{eq3.8yy}
\abs{\vec \cG(t,x,s,y)}\leq C\,\chi_{(0,\infty)}(t-s)\cdot \delta(X,Y)^{\mu_0}  \left\{(t-s)\wedge R_{max}^2\right\}^{-n/2}\exp\left(-\kappa \frac{\abs{x-y}^2}{t-s} \right),
\end{equation}
where $C=C(n,m,\nu, \mu_0,N_1)$ and $\kappa=\kappa(\nu)$.
\end{theorem}

\begin{remark}							\label{rmk:P-03}
Suppose the operator $\cL$ satisfies the following property, which we shall refer to as the condition $\LHP$:
There exist $\mu_0\in (0,1]$, $R_{max} \in (0,\infty]$, and  $C_0>0$ such that for all $X\in \cU$ and $0<R<R_{max}$, the following holds:
\begin{enumerate}[i)]
\item
If  $\vec u$ is a weak solution of $\cL \vec u=0$ in $\cU_R^-(X)$ vanishing on $\cS_R^-(X)$, then we have
\[
\int_{\cU_\rho^-(X)} \abs{D\vec u}^2\leq C_0 \left(\frac{\rho}{r}\right)^{n+2\mu_0}\int_{\cU_r^-(X)}\abs{D\vec u}^2, \quad \forall 0<\rho<r\leq R.
\]
\item
If  $\vec u$ is a weak solution of $\cLt \vec u=0$ in $\cU_R^+(X)$ vanishing on $\cS_R^+(X)$, then we have
\[
\int_{\cU_\rho^+(X)} \abs{D\vec u}^2\leq C_0 \left(\frac{\rho}{r}\right)^{n+2\mu_0}\int_{\cU_r^+(X)}\abs{D\vec u}^2, \quad \forall 0<\rho<r\leq R.
\]
\end{enumerate}
Then, the condition $\LH$ is satisfied with $N_1=N_1(n,m,\nu,\mu_0,C_0)$ and  the same $\mu_0$ and $R_{max}$; see Lemma~\ref{lem8.1} in Appendix.
The condition $\LHP$ is reminiscent of the ``Dirichlet property (D)'', which Auscher and Tchamitchian introduced in \cite{AT} in connection with the Gaussian estimates for the heat kernel of elliptic operators with complex coefficients.
We note that the condition $\LH$ is weaker than condition $\LHP$ in general.
\end{remark}

\begin{remark}
It is clear that the estimate \eqref{eq0.1} in the introduction follow from Theorem~\ref{thm1}.
Also, we note that the estimate \eqref{eq0.2} in the introduction follows from Theorem~\ref{thm2} if $\Omega$ is bounded or $R_{max}=\infty$; see Section~\ref{sec:app} and Section~\ref{sec:h} for further discussion.
\end{remark}

\section{Some Applications of Main Results}				\label{sec:app}

\subsection{Scalar case}
Let $\Omega$ be an arbitrary open connected set in $\bR^n$ and $\cU=\bR\times \Omega$.
In the scalar case (i.e., $m=1$), both conditions $\LB$ and $\IH$ are satisfied with $R_{max}=\infty$ and $N_0=N_0(n,\nu)$; see e.g., \cite{LSU, Lieberman}.
Also, in the scalar case, the Green's matrix becomes a nonnegative scalar function.
Therefore, the following corollary is an immediate consequence of Theorem~\ref{thm1}.

\begin{corollary}           \label{cor1}
Let $\cU=\bR\times \Omega$.
If $m=1$, then the Green's function $\cG(X,Y)$ of $\cL$ in $\cU$ exists and for $X, Y\in \cU$ with $X\neq Y$,  we have
\begin{equation}            \label{eq3.11p}
\cG(t,x,s,y) \leq C\, \chi_{(0,\infty)}(t-s)\cdot (t-s)^{-n/2}\exp\left\{-\kappa \abs{x-y}^2/(t-s)\right\},
\end{equation}
where $C=C(n,\nu)$ and $\kappa=\kappa(\nu)$  are universal constants independent of $\Omega$.
\end{corollary}

\begin{remark}
Corollary~\ref{cor1} is widely known and originally due to Aronson  \cite{Aronson}.
\end{remark}

Moreover, in the scalar case, the condition $\LH$ is satisfied if the base $\Omega$ satisfies the condition $\CS$, the definition of which is given below.
In fact, if $\cL$ is a small $L^\infty$-perturbation of a diagonal system, then the condition $\LH$ is satisfied whenever the base $\Omega$ satisfies the condition $\CS$; see \S\ref{sec:AD} and Lemma~\ref{lem:G-06} below.
\begin{CEM}
There exist $\theta>0$ and $R_a\in (0,\infty]$ such that
\[
\abs{B_R(x)\setminus\Omega}\ge \theta\abs{B_R(x)},\quad\forall x \in\partial\Omega,\quad \forall R< R_a.
\]
\end{CEM}

The following corollary is then an easy consequence of Theorem~\ref{thm2} and Corollary~\ref{cor1}.

\begin{corollary}						\label{cor2}
Let $\cU=\bR\times \Omega$, where $\Omega$ satisfies the condition $\CS$.
If $m=1$, then the Green's function $\cG(X,Y)$ of $\cL$ in $\cU$ exists and satisfies the estimate \eqref{eq3.11p}.
Denote
\begin{equation*}			
\delta(X,Y)= \left(1 \wedge \frac {d_X} {R_a\wedge \abs{X-Y}_p} \right) \left(1 \wedge \frac{d_Y} {R_a\wedge \abs{X-Y}_p}\right).
\end{equation*}
Then for $X, Y\in \cU$ with $X \neq Y$,  we also have
\begin{equation*}						
\cG(t,x,s,y) \leq C\,\chi_{(0,\infty)}(t-s)\cdot \delta(X,Y)^{\mu_0}  \left\{(t-s)\wedge R_a^2\right\}^{-n/2}\exp\left(-\kappa \frac{\abs{x-y}^2}{t-s} \right),
\end{equation*}
where $C=C(n,\nu,\theta)$, $\mu_0=\mu_0(n,\nu,\theta)$, and $\kappa=\kappa(\nu)$.
\end{corollary}

\begin{example}
$\bR_{+}^n$ satisfies the condition $\CS$ with $\theta=1/2$ and $R_a=\infty$.
Then we have
\[
\delta(X,Y)= \left(1 \wedge \frac {d_X} {\abs{X-Y}_p} \right) \left(1 \wedge \frac{d_Y} {\abs{X-Y}_p}\right)\quad\text{in }\; \cU=\bR\times \bR^n_+,
\]
and by Corollary~\ref{cor2}, for all $X, Y \in \cU$ with $X\neq Y$, we have
\[
\cG(t,x,s,y) \leq C\,\chi_{(0,\infty)}(t-s)\cdot \delta(X,Y)^{\mu_0} (t-s)^{-n/2}\exp\left(-\kappa \frac{\abs{x-y}^2}{t-s} \right),
\]
where $C=C(n,\nu)$, $\mu_0=\mu_0(n,\nu)$, and $\kappa=\kappa(\nu)$.
\end{example}

\begin{remark}
In Corollary \ref{cor2}, one can allow for lower order terms in $\cL$; c.f. Corollary~\ref{cor4.15} below.
\end{remark}

\subsection{$L^\infty$-perturbation of diagonal systems}        \label{sec:AD}
Let $a^{\alpha\beta}(X)$ be scalar functions satisfying
\begin{equation}            \label{eqP-07}
a^{\alpha\beta}(X)\xi_\beta\xi_\alpha\ge \nu_0\bigabs{\vec \xi}^2,\quad\forall\xi\in\bR^n;\qquad \sum_{\alpha,\beta=1}^n \bigabs{a^{\alpha\beta}(X)}^2\le \nu_0^{-2},
\end{equation}
for all $X \in\bR^{n+1}$ with some constant $\nu_0\in (0,1]$.
Let $\cU=\bR\times\Omega$ with the base $\Omega\subset \bR^n$ satisfying the condition $\CS$.
Let $A^{\alpha\beta}_{ij}$ be the coefficients of the operator $\cL$. We denote
\begin{equation}					\label{eqP-08w}
\mathscr{E}= \sup_{X\in \bR^{n+1}}\left\{
 \sum_{i,j=1}^m \sum_{\alpha,\beta =1}^n \Bigabs{A^{\alpha\beta}_{ij}(X)-a^{\alpha\beta}(X)\delta_{ij}}^2\right\}^{1/2},
\end{equation}
where $\delta_{ij}$ is the usual Kronecker delta symbol.

By Lemma~\ref{lem:G-06}, there exists $\mathscr{E}_0=\mathscr{E}_0(n,\nu_0)$ such that if
$\mathscr{E}<\mathscr{E}_0$, then the condition $\LH$ is satisfied with $\mu_0=\mu_0(n,\nu_0, \theta)$, $R_{max}=R_a$, and $N_1=N_1(n,m,\nu_0, \theta)$.
Therefore, the following corollary is another easy consequence of Theorem~\ref{thm2}.

\begin{corollary}           \label{cor2b}
Assume that $a^{\alpha\beta}(X)$ satisfy the condition \eqref{eqP-07}.
Let $\cU=\bR\times\Omega$, where $\Omega$ satisfies the condition $\CS$, and define $\delta(X,Y)$ as in \eqref{eq3.6mm} with $R_{max}=R_a$.
Let $\mathscr{E}$ be defined as in \eqref{eqP-08w}, where $A^{\alpha\beta}_{ij}(X)$ are the coefficients of the operator $\cL$.
There exists $\mathscr{E}_0=\mathscr{E}_0(n,\nu_0,\theta)$ such that if $\mathscr{E} <\mathscr{E}_0$, then the Green's matrix $\vec \cG(X,Y)$ of $\cL$ in $\cU$ exists and for all $X, Y \in \cU$ with $X\neq Y$, we have
\[
\abs{\vec \cG(t,x,s,y)}\leq C\,\chi_{(0,\infty)}(t-s)\cdot \delta(X,Y)^{\mu_0}  \left\{(t-s)\wedge R_a^2\right\}^{-n/2}\exp\left(-\kappa \frac{\abs{x-y}^2}{t-s} \right),
\]
where $C, \mu_0$, and $\kappa$ are constants depending on $n,m,\nu_0$, and $\theta$.
\end{corollary}

\begin{example}     \label{ex4.11}
Let $\Omega=\{ x\in \bR^n: x_n>\varphi(x')\}$, where $x=(x',x_n)$ and  $\varphi:\bR^{n-1}\to \bR$ is a Lipschitz function with a Lipschitz constant $K$. Then $\Omega$ satisfies the condition $\CS$ with $\theta=\theta(n,K)$ and $R_a=\infty$ and we have
\[
\delta(X,Y)= \left(1 \wedge \frac {d_X} {\abs{X-Y}_p} \right) \left(1 \wedge \frac{d_Y} {\abs{X-Y}_p}\right)\quad\text{in }\; \cU=\bR\times \Omega.
\]
If $\cL$ is a small $L^\infty$-perturbation of a diagonal system in the sense of Corollary~\ref{cor2b}, then the Green's matrix $\vec \cG(t,x,s,y)$ of $\cL$ in $\cU$  exists, and for all $X, Y \in \cU$ with $X\neq Y$, we have
\[
\abs{\vec \cG(t,x,s,y)}\leq  C\,\chi_{(0,\infty)}(t-s)\cdot \delta(X,Y)^{\mu_0} (t-s)^{-n/2}\exp\left(-\kappa \frac{\abs{x-y}^2}{t-s} \right),
\]
where $C, \mu_0$, and $\kappa$ are constants depending on $n,m,\nu_0$, and $K$.
\end{example}

\subsection{Systems with $\VMO_x$ coefficients}      \label{sec:VMO}
For a measurable function $f=f(X)=f(t,x)$ defined on $\bR^{n+1}$, we set for $\rho>0$
\[
\omega_\rho(f):=\sup_{X\in\bR^{n+1}}\sup_{r \leq \rho} \fint_{t-r^2}^{t+r^2}\!\!\!\fint_{B_r(x)} \bigabs{f(y,s)-\bar f_{x,r}(s)}\,dy\,ds;\quad \bar f_{x,r}(s)=\fint_{B_r(x)}f(s,\cdot).
\]
We say that $f$ belongs to $\VMO_x$ if $\lim_{\rho\to 0} \omega_\rho(f)=0$.
Note that $\VMO_x$ is a strictly larger class than the classical $\VMO$ space.
In particular, $\VMO_x$ contains all functions uniformly continuous in $x$ and measurable in $t$; see \cite{Krylov}.

If  $\cU=\bR\times\Omega$, where the base $\Omega$ is a bounded $C^1$ domain and if the coefficients $\vec A^{\alpha\beta}$ of the operator $\cL$ are functions in $\VMO_x$ satisfying the conditions \eqref{eqP-02} and \eqref{eqP-03}, then the condition $\LH$ is satisfied with parameters $\mu_0$, $N_1$, and $R_{max}$ depending on $\Omega$ and $\omega_\rho(\vec A^{\alpha\beta})$ as well as on $n, m, \nu$. Therefore, we have the following corollary of Theorem~\ref{thm2}.

\begin{corollary}           \label{cor3}
Let $\cU=\bR\times \Omega$, where $\Omega$ is a bounded $C^1$ domain.
Assume that the coefficients $\vec A^{\alpha\beta}$ of $\cL$ belong to $\VMO_x$ and satisfy the conditions \eqref{eqP-02} and \eqref{eqP-03}.
Then, the Green's matrix $\vec \cG(X,Y)$ of $\cL$ in $\cU$ exists and for all $X, Y \in \cU$ with $X\neq Y$ and for all $\mu_0\in (0,1)$, we have
\[
\abs{\vec \cG(t,x,s,y)}\leq C\,\chi_{(0,\infty)}(t-s)\cdot \delta(X,Y)^{\mu_0}  \left\{(t-s)\wedge R_{max}^2\right\}^{-n/2}\exp\left(-\kappa \frac{\abs{x-y}^2}{t-s} \right),
\]
where $\delta(X,Y)$ is defined as in \eqref{eq3.6mm}, and $R_{max}$, $C$, and $\kappa$ are positive constants depending on $\Omega$ and $\omega_\rho(\vec A^{\alpha\beta})$ as well as on $n$, $m$, $\nu$, and $\mu_0$.
\end{corollary}

In the above corollary, one may assume that $\vec A^{\alpha\beta}$ satisfy the weaker Legendre-Hadamard condition and may even include lower order terms in the operator. More precisely, let
\begin{equation}            \label{eq4.15ip}
\cL_\lambda \vec u = \vec u_t -D_\alpha (\vec{A}^{\alpha\beta} D_\beta \vec u)+D_\alpha(\vec{B}^\alpha \vec u)+\hat{\vec{B}}{}^\alpha D_\alpha \vec u+ \vec{C} \vec u + \lambda \vec u,
\end{equation}
where $\vec A^{\alpha\beta}, \vec B^\alpha, \hat{\vec B}{}^\alpha$, and $\vec C$ are $m\times m$ matrix valued functions on $\bR^{n+1}$ satisfying
\begin{equation}            \label{eq4.16ys}
\left\{\,\,
\begin{aligned}
A^{\alpha\beta}_{ij}(X)\xi^j \xi^i \eta_\beta \eta_\alpha \ge \nu \bigabs{\vec \xi}^2 \bigabs{\vec \eta}^2,\quad \forall\vec\xi\in \bR^m,\,\,\,\forall \vec\eta\in\bR^n, \,\,\,\forall X\in\bR^{n+1};\\
\sum_{\alpha, \beta=1}^n\bignorm{\vec A^{\alpha\beta}}_{L^\infty}^2 \leq \nu^{-2};\quad \sum_{\alpha=1}^n \left(\bignorm{\vec B^\alpha}_{L^\infty}^2 + \bignorm{\hat{\vec B}{}^\alpha}_{L^\infty}^2\right)+\bignorm{\vec C}_{L^\infty}^2\leq \nu^{-2},
\end{aligned}
\right.
\end{equation}
for some constant $\nu\in (0,1]$, and $\lambda$ is a scalar constant.

\begin{corollary}           \label{cor4.15}
Assume $\cU=\bR\times \Omega$, where $\Omega$ is a bounded $C^1$ domain.
Let the operator $\cL_\lambda$ be as in \eqref{eq4.15ip} with coefficients satisfying the condition \eqref{eq4.16ys}.
We assume further that the leading coefficients $\vec A^{\alpha\beta}$ belong to $\VMO_x$.
There exists $\lambda_0\geq 0$ such that if $\lambda> \lambda_0$, then the Green's matrix $\vec \cG(X,Y)$ of $\cL_\lambda$ in $\cU$ exists and for all $X, Y \in \cU$ with $X\neq Y$ and for all $\mu_0\in (0,1)$, we have
\[
\abs{\vec \cG(t,x,s,y)}\leq C\,\chi_{(0,\infty)}(t-s)\cdot \delta(X,Y)^{\mu_0}  \left\{(t-s)\wedge R_{max}^2\right\}^{-n/2}\exp\left(-\kappa \frac{\abs{x-y}^2}{t-s} \right),
\]
where $\delta(X,Y)$ is defined as in \eqref{eq3.6mm}, and $R_{max}$, $C$, and $\kappa$ are positive constants depending on $\Omega$, $\omega_\rho(\vec A^{\alpha\beta})$, and $\lambda$ as well as on $n$, $m$, $\nu$, and $\mu_0$.
\end{corollary}

We give a sketch of proof for Corollary~\ref{cor4.15}. First we note that for sufficiently large $\lambda$, one has the solvability of the following problem in the function space $\rV^{0,1}_2((s,\infty)\times\Omega)^m$:
\[
\left\{\begin{array}{l l}
\cL_\lambda \vec u=\vec f\\
\vec u(s,\cdot)= 0,\end{array}\right.
\]
where $\vec f\in L^\infty_c(\cU)$.
In particular, one can construct the averaged Green's matrix $\vec \cG^\rho(X,Y)$ of $\cL_\lambda$ in $\cU$ by following the argument in \cite[\S 4]{CDK}.
We also note that the condition $\LH$ is satisfied in this case; see e.g., \cite{DK09b}.
Then by modifying the proofs of Theorem~\ref{thm1} and \ref{thm2}, one can prove the above corollary. The details are left the the reader.

\begin{remark}
In Corollary \ref{cor3} and Corollary \ref{cor4.15}, the conditions of $\Omega$ and $\vec A^{\alpha\beta}$ can be relaxed. We may assume that $\Omega$ is a bounded Lipschitz domain with a sufficiently small Lipschitz constant, and $\omega_\rho(\vec A^{\alpha\beta})$ is also sufficiently small for some $\rho>0$;  see e.g., \cite{DK09b}.
\end{remark}

\section{Proofs of Main Theorems}							\label{sec:p}

\subsection{Proof of Theorem~\ref{thm1}}
By Lemma~\ref{lem8.0a} and \cite[Theorem~2.7]{CDK}, the condition $\IH$ implies the existence of the Green's matrix $\vec \cG(X,Y)$ of $\cL$ in $\cU$.
In fact, we point out that in the proof of \cite[Theorem~2.7]{CDK}, one can completely replace the property $\PH$ by the condition $\IH$, the latter of which is weaker.
This observation is useful because in the presence of lower order terms in the operator $\cL$, the property $\PH$ does not follow from the condition $\IH$.
Notice from \cite[Theorem~2.7]{CDK} that we have
\begin{equation*}					
\vec \cG(t,x,s,y)=0\quad\text{if }\,t<s.
\end{equation*}

Therefore, to prove estimate \eqref{eq2.17d}, we only need to consider the case when $t>s$.
To derive the estimate \eqref{eq2.17d}, we modify the proof in \cite[\S5.1]{CDK}, which was based on that in \cite{HK04}.
We mention that the method in \cite{HK04} was in turn based on the ideas appeared in \cite{Davies} and \cite{FS}.

Let $\psi$ be a bounded Lipschitz function on $\mathbb{R}^n$ satisfying $\abs{D \psi} \leq \gamma$ a.e. for some $\gamma>0$ to be chosen later.
For $t>s$, we define an operator $P^\psi_{s\to t}$ on $L^2(\Omega)^m$ as follows:
For a given $\vec f\in L^2(\Omega)^m$, let $\vec u$ be the unique weak solution in $\rV^{0,1}_2((s,\infty)\times\Omega)^m$ of the problem (see \cite[Lemma~2.1]{CDK})
\begin{equation} 		\label{eq3.6.1}
\left\{\begin{array}{l l}
\cL \vec u=0,\\
\vec u(s,\cdot)= e^{-\psi}\vec f
\end{array}\right.
\end{equation}
and define $P^\psi_{s\to t} \vec f(x):= e^{\psi(x)}\vec u(t,x)$.
Then, by \cite[Theorem~2.7]{CDK}, we find
\begin{equation}  \label{eq3.60.3}
P^\psi_{s\to t}\vec f(x)=
e^{\psi(x)}\int_{\Omega}\vec \cG(t,x,s,y)e^{-\psi(y)}\vec f(y)\,dy, \quad\forall \vec f\in L^2(\Omega)^m.
\end{equation}

Then, as in \cite[\S5.1]{CDK}, we derive
\begin{equation} \label{eq3.70}
\bignorm{P^\psi_{s\to t}\vec f}_{L^2(\Omega)} \leq e^{\vartheta\gamma^2(t-s)}\bignorm{\vec f}_{L^2(\Omega)}, \quad \forall t>s,
\end{equation}
where $\vartheta=\nu^{-3}$.
We set $\rho= \sqrt{t-s} \wedge R_{max}$ and use the condition $\LB$ to estimate
\begin{align*}
e^{-2\psi(x)}\bigabs{P^\psi_{s\to t}\vec f(x)}^2  &= \bigabs{\vec u(t,x)}^2\\
&\leq N_0^2 \rho^{-(n+2)} \int_{t-\rho^2}^t \int_{\Omega_{\rho}(x)}\bigabs{\vec u(\tau,y)}^2\,dy\,d\tau
\\
& \le N_0^2 \rho^{-(n+2)} \int_{t-\rho^2}^t \int_{\Omega_{\rho}(x)}e^{-2\psi(y)} \bigabs{P^\psi_{s\to\tau}\vec f(y)}^2\,dy\,d\tau,
\end{align*}
where  $\Omega_\rho(x):=\Omega\cap B_\rho(x)$.
Thus, by using \eqref{eq3.70}, we derive
\begin{align*}
\bigabs{P^\psi_{s\to t}\vec f(x)}^2 &\le N_0^2 \rho^{-n-2} \int_{t-\rho^2}^t \int_{\Omega_{\rho}(x)}e^{2\psi(x)-2\psi(y)} \bigabs{P^\psi_{s\to\tau}\vec f(y)}^2\,dy\,d\tau\\
& \le N_0^2 \rho^{-n-2} \int_{t-\rho^2}^t \int_{\Omega_{\rho}(x)}e^{2 \gamma \rho }
\bigabs{P^\psi_{s\to\tau}\vec f(y)}^2\,dy\,d\tau\\
& \le N_0^2 \rho^{-n-2} \, e^{2\gamma \rho} \int_{t-\rho^2}^t e^{2\vartheta\gamma^2(\tau-s)}\norm{\vec f}_{L^2(\Omega)}^2\,d\tau\\
& \le N_0^2 \rho^{-n}\,e^{2\gamma\rho+2\vartheta\gamma^2(t-s)}\norm{\vec f}_{L^2(\Omega)}^2.
\end{align*}
We have thus obtained the following $L^2\to L^\infty$ estimate for $P^\psi_{s\to t}$:
\begin{equation} \label{eq3.70.3}
\bignorm{P^\psi_{s\to t}\vec f}_{L^\infty( \Omega)} \leq  N_0 \rho^{-n/2}\,e^{\gamma\rho+\vartheta\gamma^2(t-s)} \bignorm{\vec f}_{L^2(\Omega)}.
\end{equation}

We also define the operator $Q^\psi_{t\to s}$ on $L^2(\Omega)^m$ for $s<t$ by setting $Q^\psi_{t\to s}\vec g(y)= e^{-\psi(y)}\vec v(s,y)$, where $\vec v$ is the unique weak solution in $\rV^{0,1}_2((-\infty,t)\times \Omega)^m$ of the backward problem
\begin{equation} \label{eq3.6.11}
\left\{\begin{array}{l l}
\cLt \vec v=0,\\
\vec v(t,\cdot)= e^{\psi}\vec g.
\end{array}\right.
\end{equation}
By a similar calculation that leads to \eqref{eq3.70.3}, we obtain
\begin{equation} \label{eq3.73.3}
\bignorm{Q^\psi_{t\to s}\vec g}_{L^\infty(\Omega)} \leq  N_0 \rho^{-n/2}\,e^{\gamma\rho+\vartheta\gamma^2(t-s)} \bignorm{\vec g}_{L^2(\Omega)}.
\end{equation}
Notice that it follows from \eqref{eq3.6.1} and \eqref{eq3.6.11} that
\[
\int_{\Omega}\bigl(P^\psi_{s\to t}\vec f\bigr) \cdot \vec g= \int_{\Omega}\vec f\cdot \bigl(Q^\psi_{t\to s} \vec g\bigr).
\]
Therefore, by duality, \eqref{eq3.73.3} implies that for all $\vec f, \vec g \in L^\infty_c(\Omega)^m$,  we have
\begin{equation} \label{eq3.74.3}
\bignorm{P^\psi_{s\to t}\vec f}_{L^2(\Omega)} \leq  N_0 \rho^{-n/2}\,e^{\gamma\rho+\vartheta\gamma^2(t-s)} \bignorm{\vec f}_{L^1(\Omega)}.
\end{equation}
Now,  set $r=(s+t)/2$ and observe that by the uniqueness, we have
\[
P^\psi_{s\to t}\vec f= P^\psi_{r\to t}(P^\psi_{s\to r}\vec f),\quad \forall \vec f\in L^\infty_c(\Omega)^m.
\]
Then, by noting that $t-r=r-s=(t-s)/2$ and $\rho/\sqrt{2}\le \sqrt{t-r}\wedge R_{\max}\le\rho$, we obtain from \eqref{eq3.70.3} and \eqref{eq3.74.3} that
\[
\bignorm{P^\psi_{s\to t}\vec f}_{L^\infty(\Omega)} \leq C \rho^{-n}\,e^{ 2\gamma\rho+\vartheta\gamma^2(t-s)} \bignorm{\vec f}_{L^1(\Omega)}, \quad \forall \vec f\in L^\infty_c(\Omega)^m,
\]
where $C=2^{n/2}N_0^2$.
For all $x, y\in \Omega$ with $x\neq y$, the above estimate combined with \eqref{eq3.60.3} yields, by duality, that
\begin{equation} \label{eq3.83}
e^{\psi(x)-\psi(y)}\abs{\vec \cG(t,x,s,y)} \leq  C \rho^{-n}\,e^{ 2\gamma\rho+\vartheta\gamma^2(t-s)}.
\end{equation}

Let $\psi(z):=\gamma\psi_0(\abs{z-y})$, where $\psi_0$ is defined on $[0,\infty)$ by
\[
\psi_0(r)=\begin{cases}r&\text{if $r \leq \abs{x-y}$} \\
\abs{x-y}&\text{if  $r>\abs{x-y}$}.
\end{cases}
\]
Then, $\psi$ is a bounded Lipschitz function satisfying $\abs{D\psi} \leq \gamma$ a.e.
Take $\gamma=\abs{x-y}/2\vartheta(t-s)$ and set $\xi:=\abs{x-y}/\sqrt{t-s}$.
By \eqref{eq3.83} and the obvious inequality $\rho/\sqrt{t-s}\leq 1$, we have
\[
\abs{\vec \cG (t,x,s,y)}\leq  C\rho^{-n}\, \exp\{\xi/\vartheta-\xi^2/4\vartheta\}.
\]
Let $N=N(\vartheta)=N(\nu)$ be chosen so that
\[
\exp(\xi/\vartheta-\xi^2/4\vartheta)\le N\exp(-\xi^2/8\vartheta),\quad\forall \xi\in [0,\infty).
\]
If we set $\kappa=1/8\vartheta=\nu^3/8$, then we obtain
\[
\abs{\vec \cG(t,x,s,y)}\le C\rho^{-n}\exp\left\{-\kappa\abs{x-y}^2/(t-s)\right\}.
\]
where $C=C(n,m,\nu,N_0)>0$.  We have thus proved the estimate \eqref{eq2.17d}.
\hfill\qedsymbol


\subsection{Proof of Theorem~\ref{thm1c}}
As mentioned in the proof of Theorem~\ref{thm1}, the condition $\IH$ implies the existence of the Green's matrix $\vec \cG(X,Y)$ of $\cL$ in $\cU$.
Moreover, the Green's matrix ${}^t \vec \cG(X,Y)$ of $\cLt$ in $\cU$ also exists and we have the following identity:
\begin{equation}		\label{eq2.20f}
\vec \cG(X,Y)={}^t\vec \cG(Y,X)^T,\quad \forall X,Y\in \cU,\quad X\neq Y.
\end{equation}

With aid of the above observation, we shall prove below that \eqref{eq2.17dd} implies part  ii) in the condition $\LB$. Thanks to \eqref{eq2.20f}, the proof that \eqref{eq2.17dd} implies part i) in $\LB$ is similar and shall be omitted.
Notice that \eqref{eq2.17dd} implies, via straightforward computation, that
\begin{equation}					\label{eq5.18}
\abs{\vec \cG(X,Y)} \leq C \abs{X-Y}_p^{-n},\quad \text{if }\, 0<\abs{t-s} <R_{max}^2,
\end{equation}
where $C=C(n,m,\nu,C_0, \kappa)$.
Then, by the energy inequality (see \cite[Eq. (3.21)]{CDK}) and \eqref{eqn:2.2}, and observing the obvious fact that $R_{max}/2$ and $R_{max}$ are comparable to each other in the case when $R_{max}<\infty$, we obtain
\begin{equation} \label{eq5.15}
\norm{\vec \cG(\cdot,Y)}_{L^{2+4/n}(\cU\setminus \overline Q_r(Y))}+\tri{\vec \cG(\cdot,Y)}_{\cU\setminus \overline Q_r(Y)} \leq Cr^{-n/2}, \quad \forall r \in (0,R_{max}),
\end{equation}
where $C=C(n,m,\nu,C_0,\kappa)$.

Let $X\in\cU$ and $R\in (0,R_{max})$ be given.
Without loss of generality, we may assume $X=0$ and for simplicity of notation, we shall write $\cU_R^\pm=\cU^\pm_R(0)$, etc.

Assume that $\vec u$  is a weak solution of $\cLt\vec u=0$ in $\cU^+_R$ vanishing on $\cS^+_R$.
Let $\vec w=\zeta \vec u$, where $\zeta$ is a smooth cut-off function on $\bR^{n+1}$ satisfying
\begin{equation}				\label{eq5.9z}
0\leq \zeta \leq 1,\,\,\, \supp \zeta \subset Q_{R/2},\,\,\,\zeta \equiv 1\,\text{ on }\, Q_{3R/8},\,\,\, \abs{D\zeta} \le 16/R,\,\,\, \text{and}\,\,\, \abs{\zeta_t} \le 16/R^2 .
\end{equation}
Then $\vec w$ is a weak solution of
\begin{equation*}				
\cLt \vec w= - \zeta_t \vec u - {}^t\!\vec A^{\alpha\beta} D_\alpha \zeta D_\beta \vec u - D_\alpha({}^t\!\vec A^{\alpha\beta} D_\beta \zeta \vec u)
\end{equation*}
in $\cU^+:=\bR_+\times\Omega$. Notice that $\vec w$ vanishes on $\bR_+\times \partial\Omega$ and $\vec w\equiv 0$ in $(R^2/4,\infty)\times \Omega$.

For $Y \in \cU_{R/4}^+$, let $\rho>0$ be such that $Q_\rho^-(Y)\subset \cU_R^+$.
For $k=1,\ldots, m$,  let $\vec v_\rho$ be the $k$-th column of $\vec \cG^\rho(\cdot,Y)$, where $\vec \cG^\rho(\cdot,Y)$ is the averaged Green's matrix of $\cL$ in $\cU$ as constructed in \cite[\S4.1]{CDK}.
Then, similar to \cite[Eq. (3.8)]{CDK}, we have
\begin{align}					\label{eq5.11b}
\fint_{\cU_\rho^-(Y)} \zeta u^k & = -\int_{\cU_{R/2}^+}\zeta_t \vec u\cdot\vec v_\rho-
\int_{\cU_{R/2}^+}({}^t\!\vec A^{\alpha\beta} D_\alpha \zeta D_\beta \vec u)\cdot\vec v_\rho
+\int_{\cU_{R/2}^+} ({}^t\!\vec A^{\alpha\beta} D_\beta \zeta \vec u)\cdot D_\alpha \vec v_\rho \\
\nonumber
&=:I_1+I_2+I_3,
\end{align}
which simply means that
\[ \int_{\cU^+} \vec w\cdot \cL \vec v_\rho=\int_{\cU^+} \cLt \vec w \cdot \vec v_\rho.\]
Notice from \eqref{eq5.9z} that $\abs{Y-Z}_p> R/8$ for $Z\in \cU_R^+\cap \supp D\zeta$, and recall ${}^t\!\vec A^{\alpha\beta}=(\vec A^{\beta\alpha})^T$.
Thus, if we set $r=R/8\wedge (d_Y \wedge R_c)$, then after interchanging indices $\alpha$ and $\beta$, we find that
\[
I_2+I_3= - \int_{\cU_{R/2}^+ \setminus Q_r(Y)} A^{\alpha\beta}_{ij}  \cG^\rho_{jk}(\cdot,Y)D_\alpha u^i D_\beta\zeta+\int_{\cU_{R/2}^+\setminus Q_r(Y)} A^{\alpha\beta}_{ij} D_\beta \cG^\rho_{jk}(\cdot,Y)u^i D_\alpha \zeta.
\]

Then, in light of \cite[Eq. (4.15)]{CDK} and \cite[Eq. (4.16)]{CDK}, we  take limits $\rho$ to zero in \eqref{eq5.11b} and conclude that for a.e. $Y\in \cU^+_{R/4}$, we have
\begin{multline}				\label{eq4.28v}
u^k(Y) = - \int_{\cU_{R/2}^+}  \zeta_t  u^j  \cG_{jk}(\cdot,Y)- \int_{\cU_{R/2}^+} A^{\alpha\beta}_{ij}  \cG_{jk}(\cdot,Y)D_\alpha u^i D_\beta\zeta\\
+\int_{\cU_{R/2}^+} A^{\alpha\beta}_{ij} D_\beta  \cG_{jk}(\cdot,Y)u^iD_\alpha \zeta
=:  I_1'+ I_2'+ I_3'.
\end{multline}
Let $\cA_R(Y):=\cU_{3R/4}(Y)\setminus Q_{R/8}(Y) \supset \cU^+_{R/2}\setminus Q_{R/8}(Y)$.
By \eqref{eq5.9z}, H\"older's inequality, and \eqref{eq5.15}, we estimate
\begin{align*}
\bigabs{I_1'}  & \leq C R^{-2} \norm{\vec \cG(\cdot,Y)}_{L^2(\cA_R(Y))}\, \norm{\vec u}_{L^2(\cU^+_{R/2})} \leq C R^{-(n+2)/2} \norm{\vec u}_{L^2(\cU^+_{R})}.\\
\intertext{Similarly, by \eqref{eq5.9z} and \eqref{eq5.15}, we obtain}
\bigabs{I_3'}  & \leq C R^{-1} \norm{D\vec \cG(\cdot,Y)}_{L^2(\cA_R(Y))}\, \norm{\vec u}_{L^2(\cU^+_{R/2})} \leq C R^{-(n+2)/2} \norm{\vec u}_{L^2(\cU^+_{R})}.
\end{align*}

Notice that  the energy inequality (see e.g., \cite[\S III.2]{LSU}) yields
\begin{equation}				\label{eq5.17tc}
\norm{D \vec u}_{L^2(\cU_{R/2}^+)} \leq C R^{-1}\norm{\vec u}_{L^2(\cU_{R}^+)}.
\end{equation}
By \eqref{eq5.9z}, \eqref{eq5.18}, and \eqref{eq5.17tc}, we estimate
\[
\bigabs{I_2'}  \leq C R^{-1} \norm{\vec \cG(\cdot,Y)}_{L^2(\cA_R(Y))}\, \norm{D \vec u}_{L^2(\cU^+_{R/2})} \leq  C R^{-(n+2)/2} \norm{\vec u}_{L^2(\cU^+_{R})}.
\]

By combining above estimates for $I_1', I_2'$, and $I_3'$, we conclude from \eqref{eq4.28v} that
\[
\norm{\vec u}_{L^\infty(\cU^+_{R/4})} \leq C R^{-(n+2)/2} \norm{\vec u}_{L^2(\cU^+_R)},
\]
where $C=C(n,m,\nu,C_0, \kappa)$.
Since the above estimate holds for all $X\in\cU$ and $R\in(0,R_{max})$, we obtain $\LB$ by a standard covering argument.
\hfill\qedsymbol


\subsection{Proof of Theorem~\ref{thm2}}
Notice that by Lemma~\ref{lem2.19} and Theorem~\ref{thm1}, we have
\begin{equation}			\label{eq2.25t}
\abs{\vec \cG(t,x,s,y)}\leq C_0\, \chi_{(0,\infty)}(t-s) \cdot \left\{(t-s)\wedge R_{max}^2\right\}^{-n/2}\exp\left\{-\kappa\abs{x-y}^2/(t-s)\right\},
\end{equation}
where $C_0=C_0(n,m,\nu, \mu_0,N_1)$.
We denote
\[
\delta_1(X,Y)= \left(1 \wedge \frac {d_X} {R_{max}\wedge \abs{X-Y}_p} \right)\quad\text{and}\quad
\delta_2(X,Y)= \left(1 \wedge \frac{d_Y} {R_{max}\wedge \abs{X-Y}_p}\right)
\]
so that $\delta(X,Y)=\delta_1(X,Y)\cdot \delta_2(X,Y)$.
To prove the estimate \eqref{eq3.8yy}, we first claim that
\begin{equation}			\label{eq22.00k}
\abs{\vec \cG(t,x,s,y)}\leq C\chi_{(0,\infty)}(t-s) \cdot \delta_1(X,Y)^{\mu_0} \left\{(t-s)\wedge R_{max}^2\right\}^{-n/2}\exp\left\{-\frac{\kappa \abs{x-y}^2}{4(t-s)}\right\},
\end{equation}
where $C=C(n,m,\nu, \mu_0,N_1)$.
The following lemma is a key to prove the above claim.

\begin{lemma}			\label{lem3.6}
Let $\cU=\bR\times\Omega$ and assume the condition $\LH$.
For $R \in (0,R_{max})$ and $X\in\cU$ such that $d_X< R/2$,  let $\vec u$ be a weak solution of $\cL\vec u=0$ in $\cU_R^-(X)$ vanishing on $\cS_R^-(X)$.
Then, we have
\begin{equation}			\label{eq3.7m}
\abs{\vec u(X)} \le C d_X^{\mu_0} R^{-n/2-1-\mu_0}\norm{\vec u}_{L^2(\cU_R^-(X))},
\end{equation}
where $C=C(n,m,\nu, \mu_0, N_1)$.
\end{lemma}

\begin{proof}
By the very definition the condition $\LH$, we have
\begin{equation}			\label{eq3.8x}
\bigabs{\tilde{\vec u}(X')-\tilde{\vec u}(X)} \leq C \abs{X'-X}_p^{\mu_0}\, R^{-n/2-1-\mu_0}\norm{\vec u}_{L^2(\cU_R^-(X))},\quad\forall X'\in Q_{R/2}^-(X).
\end{equation}
For $r\in (d_X, R/2)$, there is $X'  \in Q_{R/2}^-(X) \setminus \cU$ such that $\abs{X-X'}_p=r$.
By \eqref{eq3.8x} we obtain
\[
\bigabs{\vec u(X)}=\bigabs{\tilde{\vec u}(X)-\tilde{\vec u}(X')} \leq C r^{\mu_0} R^{-n/2-1-\mu_0}\norm{\vec u}_{L^2(\cU_R^-(X))}.
\]
By taking limit $r\to d_X$ in the above inequality, we derive \eqref{eq3.7m}.
\end{proof}

Now we are ready to prove \eqref{eq22.00k}.
Take $R=(R_{max}\wedge \abs{X-Y}_p)/4$.
We may assume that $d_X < R/2$ and $t>s$ because otherwise \eqref{eq22.00k} follows from \eqref{eq2.25t}.
We then set $\vec u$ to be the $k$-th column of $\vec \cG(\cdot,Y)$, for $k=1,\ldots, m$, in Lemma~\ref{lem3.6} to obtain
\begin{equation}
							\label{eq17.24}
\abs{\vec \cG(X,Y)} \leq C d_X^{\mu_0} R^{-n/2-1-\mu_0} \norm{\vec \cG(\cdot,Y)}_{L^2(\cU_{R}^-(X))};\quad R=(R_{max}\wedge \abs{X-Y}_p)/4.
\end{equation}

Next,  we consider the following three possible cases.

\begin{case1}
In this case $R=\sqrt{t-s}/4=\abs{X-Y}_p/4$ and thus, we get from \eqref{eq17.24} and \eqref{eq5.18} that
\[
\abs{\vec \cG(X,Y)} \leq C d_X^{\mu_0} R^{-n/2-1-\mu_0} \norm{\vec \cG(\cdot,Y)}_{L^2(\cU_R^-(X))}\le C d_X^{\mu_0} R^{-n-\mu_0},
\]
which immediately implies \eqref{eq22.00k} in this case.
\end{case1}

\begin{case2}
In this case $R=(\abs{x-y}\wedge R_{max})/4$.
We denote $Z=(r,z)$ and claim that for all $Z \in \cU_{2R}^-(X)$, we have
\begin{equation}			\label{eq5.26rv}
\abs{\vec \cG(r,z,s,y)} \leq CC_0(t-s)^{-n/2} \exp\left\{-\kappa \abs{x-y}^2/4(t-s)\right\},
\end{equation}
where $C_0$ and $\kappa$ are the same constants as in \eqref{eq2.25t} and $C=C(n,\kappa)$.
To prove the claim, first note that we may assume $Y=0$ without loss of generality.
Then by \eqref{eq2.25t} we have
\[
\abs{\vec \cG(r,z,s,y)} \leq C_0\, \chi_{(0,\infty)}(r)\cdot r^{-n/2} e^{-\kappa \abs{z}^2/r}  \leq C_0\,  \chi_{(0,\infty)}(r) \cdot r^{-n/2} e^{-\kappa \abs{x}^2/4r},
\]
where we used $\abs{z}=\abs{z-y} \geq \abs{x-y}/2=\abs{x}/2$.
Let us denote
\[
g(\tau)=\chi_{(0,\infty)}(\tau)\cdot \tau^{-n/2} e^{-\kappa \abs{x}^2/ 4\tau};\quad
g_0(\tau)=\chi_{(0,\infty)}(\tau)\cdot \tau^{-n/2} e^{-\kappa/4 \tau}.
\]
Then the claim \eqref{eq5.26rv} will follow if we show that there exists a positive number $C=C(n,\kappa)$ such that $g(r)<C g(t)$ for all $r <t < \abs{x}^2$, which in turn will follow if  we show that $g_0(r_1) \leq C g_0( r_2)$ for all $r_1<r_2 \leq1$.
But the latter assertion is easy to verify by an elementary analysis of the function $g_0$.

We have thus proved \eqref{eq5.26rv}, which combined with \eqref{eq17.24} yields
\[
\abs{\vec \cG(X,Y)} \leq C d_X^{\mu_0} R^{-\mu_0} (t-s)^{-n/2} \exp \left\{-\kappa \abs{x-y}^2/4(t-s)\right\}.
\]
Therefore, we also obtain \eqref{eq22.00k} in this case.
\end{case2}

\begin{case3}
In this case $R=R_{max}/4$, and the desired estimate \eqref{eq22.00k} becomes
\begin{equation}			\label{eq22.15}
\abs{\vec \cG(t,x,s,y)}\leq C\set{d_X/R_{max}}^{\mu_0} R_{max}^{-n}\exp\bigset{-\kappa\abs{x-y}^2/4(t-s)}.
\end{equation}
Since $t-s\geq 16 R^2$, for all $Z=(r,z)\in \cU_{2R}^-(X)$, we have
\begin{equation}
                        \label{eq18.21}
\exp\left\{- \kappa\, \frac{\abs{z-y}^2}{r-s}\right\}\le
\exp\left\{- \kappa\, \frac{\abs{x-y}^2/2-\abs{z-x}^2}{t-s}\right\} \leq e^{\kappa/4} \exp\left\{- \frac{\kappa\abs{x-y}^2}{2(t-s)}\right\}.
\end{equation}
Then, from \eqref{eq17.24}, \eqref{eq2.25t}, and \eqref{eq18.21}, we obtain \eqref{eq22.15}, which implies \eqref{eq22.00k} in this case.
\end{case3}

We have thus proved that the estimate \eqref{eq22.00k} holds in all possible cases.
Finally, notice that the condition $\LH$ is symmetric between $\cL$ and $\cLt$.
Therefore, by repeating the above argument to ${}^t\vec \cG(\cdot,X)$, using the identity \eqref{eq2.20f}, and utilizing the estimate \eqref{eq22.00k} instead of \eqref{eq2.25t}, we obtain \eqref{eq3.8yy} with $\kappa/16$ in place of $\kappa$.
The theorem is proved.
\hfill\qedsymbol


\section{Green's matrix for elliptic systems in two dimensional domains}				 \label{sec:2d}

In this section, we are concerned with the Green's matrix for elliptic systems
\begin{equation}							\label{eq6.00v}
L\vec u := -D_\alpha\bigl(\vec A^{\alpha\beta}(x) D_\beta \vec u\bigr)
\end{equation}
in a domain $\Omega\subset \bR^2$.
Here,  $\vec A^{\alpha\beta}=\vec A^{\alpha\beta}(x)$ are $m\times m$ matrix valued functions on $\bR^2$ with entries $A^{\alpha\beta}_{ij}(x)$ satisfying the strong ellipticity condition
\begin{equation}    \label{eqE-02}
A^{\alpha\beta}_{ij}(x)\xi^j_\beta \xi^i_\alpha \ge \nu \abs{\vec \xi}^2, \quad\forall \vec\xi \in \bR^{2m},\quad\forall x \in\bR^2,
\end{equation}
and also the uniform boundedness condition
\begin{equation}    \label{eqE-03}
\sum_{i,j=1}^m \sum_{\alpha,\beta=1}^2 \bigabs{A^{\alpha\beta}_{ij}(x)}^2\le \nu^{-2},\quad\forall x \in\bR^2,
\end{equation}
for some constant $\nu\in (0,1]$.
We emphasize that we do not impose any other conditions on the coefficients.
The adjoint operator $\Lt$ is defined by
\[
\Lt \vec u  = -D_\alpha \bigl({}^t\!\vec A^{\alpha\beta} D_\beta \vec u\bigr),\quad\text{where}\; {}^t\!\vec A^{\alpha\beta}=\bigl(\vec A^{\beta\alpha}\bigr)^T.
\]
Note that the coefficients ${}^t\!A^{\alpha\beta}_{ij}$ of $\Lt$ satisfy the conditions \eqref{eqE-02}, \eqref{eqE-03} with the same $\nu$.

Throughout in this section,  we shall always mean  by $\cL$ the operator $\partial_t - L$ on $\bR^3$; i.e.,
\begin{equation*}				
\cL \vec u = \vec u_t -D_\alpha\bigl(\vec A^{\alpha\beta}(x) D_\beta \vec u\bigr)
\end{equation*}
for functions $\vec u$ defined on the cylinder $\cU=\bR\times \Omega$.
It is obvious that the operator $\cL$ satisfies the conditions \eqref{eqP-02} and \eqref{eqP-03}.
Moreover, by \cite[Theorem~3.3]{Kim}, the condition $\IH$ is satisfied by $\cL=\partial_t-L$; see also \cite[Corollary~2.9]{CDK}.

\subsection{notation and definitions}
For $p\geq 1$ and $k$ a nonnegative integer, we denote by $W^{k,p}(\Omega)$ the usual Sobolev space; see e.g., \cite{GT}.
The function space $Y^{1,2}_0(\Omega)$ is defined as the set of all weakly
differentiable functions on $\Omega$ such that $D u\in L^2(\Omega)$ and $\eta u \in W^{1,2}_0(\Omega)$ for any $\eta\in C^\infty_c(\bR^2)$.
An open set $\Omega\subset \bR^2$ is said to be a Green domain if
\[
\{\chi_\Omega\, u: u\in C^\infty_c(\bR^2)\}\nsubset W^{1,2}_0(\Omega).
\]
We recall that if  $\Omega\subset \bR^2$ be a Green domain, then $Y^{1,2}_0(\Omega)$ is a Hilbert space when endowed with the inner product
\[
\ip{u,v}:=\int_\Omega D_i u D_i v.
\]
Moreover, $C^\infty_c(\Omega)$ is a dense subset in this Hilbert space; see e.g., \cite[\S 1.3.4]{MZ}.

For a given function $\vec f=(f^1,\ldots,f^m)^T\in L^1_{loc}(\Omega)^m$, we call $\vec u=(u^1,\ldots,u^m)^T$ a weak solution in $Y^{1,2}_0(\Omega)$ of $L\vec u=\vec f$ if $\vec u \in Y^{1,2}_0(\Omega)$ and
\begin{equation}
\label{eq:P08c}
\int_\Omega A^{\alpha\beta}_{ij} D_\beta u^j D_\alpha \phi^i= \int_\Omega f^i\phi^i,
\quad \forall \vec \phi\in C^\infty_c(\Omega)^m.
\end{equation}
It is routine to check that if $\Omega$ is a Green domain and $\vec u$ is a weak solution in $Y^{1,2}_0(\Omega)$ of $L\vec u=0$, then $\vec u\equiv 0$.
Therefore, a weak solution in $Y^{1,2}_0(\Omega)$ of $L\vec u = \vec f$ is unique.

For a Green domain $\Omega$, we say that an $m\times m$ matrix valued function $\vec G(x,y)$, with entries $G_{ij} (x,y)$ defined on the set $\bigset{(x,y)\in\Omega\times\Omega: x\neq y}$, is a Green's matrix of $L$ in $\Omega$ if it satisfies the following properties:
\begin{enumerate}[i)]
\item
$\vec G(\cdot,y)\in W^{1,1}_{loc}(\Omega)$ and $L \vec G(\cdot,y) = \delta_y I$ for all $y\in\Omega$, in the sense that
\[
\int_{\Omega}A^{\alpha\beta}_{ij}D_\beta G_{jk}(\cdot,y)D_\alpha \phi^i = \phi^k(y), \quad \forall \vec \phi \in C^\infty_c(\Omega)^m.
\]
\item
$\vec G(\cdot,y) \in Y^{1,2} (\Omega\setminus B_r(y))$ for all $y\in\Omega$ and $r>0$, and $\vec G(\cdot,y)$ vanishes on $\partial\Omega$.
\item
For any $\vec f=(f^1,\ldots, f^m)^T \in L^\infty_c(\Omega)$, the function $\vec u$ given by
\[
\vec u(x):=\int_\Omega \vec G(y,x) \vec f(y)\,dy
\]
belongs to $Y^{1,2}_0(\Omega)$ and satisfies $\Lt \vec u=\vec f$ in the sense of \eqref{eq:P08c}.
\end{enumerate}
By the remark made above, part iii) of the above definition gives the uniqueness of a Green's matrix if $\Omega$ is a Green domain.
We shall hereafter say that $\vec G(x,y)$ is ``the'' Green's matrix of $L$ in  a Green domain $\Omega$ if $\vec G(x,y)$ satisfies all the above properties.
We define
\begin{equation*}				
\varrho(\Omega):= \sqrt{\abs{\Omega}} \wedge \ell(\Omega),
\end{equation*}
where $\abs{\Omega}$ the Lebesgue measure of $\Omega$ and
\[
\ell (\Omega):=\inf\set{\dist(l_1,l_2):\Omega\,\,\text{lies between two parallel lines}\; l_1, l_2}.
\]
We remark that $\Omega$ is a Green domain if $\varrho(\Omega)<\infty$; see e.g., \cite[Lemma~3.1]{DK09}.

\subsection{Main result}

Recall that we assume that the operator $L$ in \eqref{eq6.00v} satisfies the conditions \eqref{eqE-02} and \eqref{eqE-03}, and that $\Omega$ is an open connected set in $\bR^2$.
In the sequel, we denote $d_x=\dist(x,\partial\Omega)$ and $\ln_+ t = \max(\ln t ,0)$.

\begin{theorem}						\label{thm6.1g}
Assume that $\varrho=\varrho(\Omega)<\infty$.
\begin{enumerate}[(a)]
\item
Suppose the condition $\LB$ is satisfied by $\cL=\partial_t -L$ in $\cU=\bR\times \Omega$.
Then the Green's matrix $\vec G(x,y)$ of $L$ in $\Omega$ exists and we have
\begin{equation}						\label{eq13.03k}
\abs{\vec G(x,y)} \leq C\left(\frac{\varrho}{\varrho\wedge R_{max}}\right)^3 +C\ln_+ \left(\frac{\varrho\wedge R_{max}}{\abs{x-y}}\right), \quad \forall x,y \in\Omega,\quad x\neq y,
\end{equation}
where $C=C(m, \nu, N_0)$.
\item
Suppose the condition $\LH$ is satisfied by $\cL=\partial_t -L$ in $\cU=\bR\times \Omega$.
Let $\vec G(x,y)$ be the Green's matrix of $L$ in $\Omega$.
Then for $x, y\in \Omega$ with $x\neq y$, we have
\begin{multline}					\label{eq12.01s}
\abs{\vec G(x,y)} \leq C\left(1\wedge \frac{d_x}{\abs{x-y}}\right)^{\mu_0}\left(1\wedge \frac{d_y}{\abs{x-y}}\right)^{\mu_0} \left\{ 1+ \ln_+ \left(\frac{\varrho\wedge R_{max}}{\abs{x-y}}\right)\right\}\\
+C\left(\frac{\varrho}{\varrho\wedge R_{max}}\right)^3\left(1\wedge \frac{d_x}{\varrho\wedge R_{max}}\right)^{\mu_0}\left(1\wedge \frac{d_y}{\varrho\wedge R_{max}}\right)^{\mu_0},
\end{multline}
where $C=C(m, \nu, \mu_0,N_1)$.
\end{enumerate}
\end{theorem}

\begin{remark}
We point out that the assumption $\varrho=\varrho(\Omega)<\infty$ in Theorem~\ref{thm6.1g} can be somehow relaxed. 
In fact, the quantity $\varrho$ is related to the constant $K=K(\Omega)$ in the following Poincar\'e's inequality (see \cite[Lemma~3.1]{DK09}):
\[
\norm{\varphi}_{L^2(\Omega)} \leq K \norm{\nabla \varphi}_{L^2(\Omega)},\quad\forall \varphi\in C^\infty_c(\Omega).
\]
We can replace $\varrho$ by $K$ in Theorem~\ref{thm6.1g} as long as such a constant $K<\infty$ exists.
\end{remark}

\begin{theorem}						\label{thm6.2g}
Assume that the condition $\LH$ is satisfied  by $\cL=\partial_t -L$ in $\cU=\bR\times \Omega$ with $R_{max}=\infty$.
Then the Green's matrix $\vec G(x,y)$ of $L$ in $\Omega$ exists and for all $x,y\in\Omega$ with $x\neq y$, we have
\begin{equation}						\label{eq12.01k}
\abs{\vec G(x,y)} \leq
C\left(1\wedge \frac{d_x}{\abs{x-y}}\right)^{\mu_0}\left(1\wedge \frac{d_y}{\abs{x-y}}\right)^{\mu_0} \left\{1+ \ln_+ \left(\frac{d_x\wedge d_y}{\abs{x-y}}\right)\right\},
\end{equation}
where $C=C(m, \nu, \mu_0,N_1)$.
\end{theorem}

\subsection{Applications}

Here are some easy consequences of the main theorems in this section.

\begin{corollary}						\label{cor6.3g}
Assume $m=1$ and $\varrho=\varrho(\Omega)<\infty$.
Then the Green's function $G(x,y)$ of $L$ in $\Omega$ exists and for $x, y\in \Omega$ with $x\neq y$,  we have
\[
G(x,y) \leq  C\bigset{1+ \ln_+\bigl(\varrho/\abs{x-y}\bigr)},
\]
where $C=C(\nu)$. If, in addition, $\Omega$ satisfies the condition $\CS$, then for $x\neq y$, we have
\begin{multline*}
G(x,y) \leq C\left(1\wedge \frac{d_x}{\abs{x-y}}\right)^{\mu_0}\left(1\wedge \frac{d_y}{\abs{x-y}}\right)^{\mu_0} \left\{ 1+ \ln_+ \left(\frac{\varrho\wedge R_a}{\abs{x-y}}\right)\right\}\\
+C\left(\frac{\varrho}{\varrho\wedge R_a}\right)^3\left(1\wedge \frac{d_x}{\varrho\wedge R_a}\right)^{\mu_0}\left(1\wedge \frac{d_y}{\varrho\wedge R_a}\right)^{\mu_0},
\end{multline*}
where $C=C(\nu, \theta)$.
\begin{proof}
Notice that the condition $\LB$ is satisfied by $\cL=\partial_t-L$ in any cylinder $\cU=\bR\times\Omega$ with $R_{max}=\infty$ when $m=1$; see e.g., \cite[Theorem~6.30]{Lieberman}.
Moreover, the condition $\LH$ is satisfied as well with $R_{max}=R_a$ if the base $\Omega$ enjoys the condition $\CS$; see \cite[Theorem~6.32]{Lieberman} and Lemma~\ref{lem:G-06}.
Therefore, the corollary follows from Theorem~\ref{thm6.1g}.
\end{proof}

\end{corollary}

\begin{corollary}						\label{cor6.4g}
Let $\Omega=\{ x\in \bR^2: x_2>\varphi(x_1)\}$, where $\varphi:\bR \to \bR$ is a Lipschitz function with a Lipschitz constant $K$.
Then, the Green's matrix $\vec G(x,y)$ of $L$ in $\Omega$ exists and for all $x, y\in\Omega$ with $x\neq y$, we have
\[
\abs{\vec G(x,y)} \leq
C\left(1\wedge \frac{d_x}{\abs{x-y}}\right)^{\mu_0}\left(1\wedge \frac{d_x}{\abs{x-y}}\right)^{\mu_0} \left\{1+ \ln_+ \left(\frac{d_x\wedge d_y}{\abs{x-y}}\right)\right\},
\]
where $C=C(m,\nu,K)$ and $\mu_0=\mu_0(\nu,K)$.
\end{corollary}
\begin{proof}
It is known that the condition $\LH$ is satisfied by $\cL=\partial_t -L$ in $\cU=\bR\times\Omega$ with $R_{max}=\infty$; see \cite[Lemma~4.4]{DK09}.
Therefore, the corollary follows from Theorem~\ref{thm6.2g}.
\end{proof}

\begin{corollary}         \label{cor6.5g}
Assume that $\Omega$ is a bounded Lipschitz domain.
Then the Green's matrix $\vec G(x,y)$ of $L$ in $\Omega$ exists and for all $x, y\in\Omega$ with $x\neq y$, we have
\[
\abs{\vec G(x,y)} \leq C\left(1\wedge \frac{d_x}{\abs{x-y}}\right)^{\mu_0}\left(1\wedge \frac{d_y}{\abs{x-y}}\right)^{\mu_0} \left\{ 1+ \ln_+ \left(\frac{1}{\abs{x-y}}\right)\right\},
\]
where $C=C(m, \nu, \Omega)$.
\end{corollary}
\begin{proof}
By using \cite[Lemma~4.4]{DK09} and standard partition of unity argument, one can show that the condition $\LH$ is satisfied by $\cL=\partial_t-L$.
The corollary is an easy consequence of Theorem~\ref{thm6.1g}, part (b) together with the assumption $\diam \Omega<\infty$.
\end{proof}

\subsection{Proof of Theorem~\ref{thm6.1g}}
Existence of the Green's matrix $\vec G(x,y)$ of $L$ in $\Omega$ follows from the assumption $\varrho(\Omega)<\infty$; see \cite[Theorem~2.12]{DK09}.
To be more precise, first we recall that $\IH$ is satisfied by $\cL$; see \cite{Kim}.
Let $\vec \cG(t,x,s,y)$ be the Green's matrix of $\cL$ in $\cU$ given as in Theorem \ref{thm1}. Since $\vec A^{\alpha\beta}$ are independent of $t$, we have $\vec \cG(t,x,s,y)=\vec\cG(t-s,x,0,y)$.

In the sequel, we shall denote the ``Dirichlet heat kernel'' of $L$ by
\begin{equation*}
\vec K(t,x,y)=\vec K(X,y):=\vec \cG(X,\tilde Y)= \vec \cG (t,x,0,y);\quad \tilde Y=(0,y).
\end{equation*}
It is shown in \cite{DK09} that if $\varrho<\infty$, then the Green's matrix $\vec G(x,y)$ of $L$ in $\Omega$ is given by
\begin{equation}							\label{eq:g01}
\vec G(x,y)=\int_0^\infty \vec K(t,x,y)\,dt, \quad \forall x,y \in \Omega,\quad x \neq y.
\end{equation}

\begin{parta}
From the estimates \eqref{eq5.18} and \eqref{eq5.15}, we obtain
\begin{align}
										\label{eq618z}
&\abs{\vec K(X,y)} \leq C\abs{X-\tilde Y}_p^{-2}\quad \text{if}\,\,\, 0<t< R_{max}^2; \\
										\label{eq615b}
&\norm{\vec K(\cdot,y)}_{L^{2+4/n}(\cU\setminus \overline Q_r(\tilde Y))}+ \tri{\vec K(\cdot,y)}_{\cU\setminus \overline Q_r(\tilde Y)} \leq Cr^{-1}, \quad \forall r<R_{max}.
\end{align}
By using $\LB$ and \eqref{eq615b}, and following the proof of \cite[Lemma~3.12]{DK09},  we also obtain
\begin{equation}							\label{eq12.27g}
\abs{\vec K(t,x,y)} \leq C \varrho r^{-3}e^{-2\nu (t-2r^2)/\varrho^2},\quad \forall t > 2r^2, \quad \forall r \in (0, R_{max}).
\end{equation}

Now, we are ready to prove the estimate \eqref{eq13.03k}. We set $r=(\varrho\wedge R_{max})/2$.
If $0<\abs{x-y} \leq  r$, then by \eqref{eq:g01}, we have
\begin{equation}							\label{eq6.21ee}
\abs{\vec G(x,y)} \leq \int_0^{\abs{x-y}^2} + \int_{\abs{x-y}^2}^{2r^2}+\int_{2 r^2}^\infty \abs{\vec K(t,x,y)}\,dt =: I_1+I_2+I_3.
\end{equation}
It then follows from \eqref{eq618z} and \eqref{eq12.27g} that
\begin{align*}
I_1 &\le C\int_0^{\abs{x-y}^2}\abs{x-y}^{-2}\,dt \leq C,\\
I_2&\le C\int_{\abs{x-y}^2}^{2r^2}t^{-1}\,dt\le C+C\ln (r/\abs{x-y}),\\
I_3&\le C\int_{2r^2}^\infty \varrho r^{-3}e^{-2\nu (t-2r^2)/\varrho^2}\,dt\le C r^{-3} \varrho^3.
\end{align*}
Combining the above estimates together we obtain that
\begin{equation*}							
\abs{\vec G(x,y)} \leq C\left(\frac{\varrho}{\varrho\wedge R_{max}}\right)^3 + C\ln \left(\frac{\varrho\wedge R_{max}}{2\abs{x-y}}\right),
\end{equation*}
which proves the estimate \eqref{eq13.03k} in the case when $0<\abs{x-y}\leq (\varrho\wedge R_{max})/2$.

On the other hand, if $\abs{x-y} \geq r =(\varrho\wedge R_{max})/2$, then we estimate by \eqref{eq618z} and \eqref{eq12.27g}
\begin{align*}							
\abs{\vec G(x,y)} &\leq \int_0^{2r^2}+\int_{2r^2}^\infty \abs{\vec K(t,x,y)}\,dt\\
									\nonumber
& \leq C \int_0^{2r^2} r^{-2} + C \int_{2r^2}^\infty \ \varrho r^{-3}e^{-2\nu (t-2r^2)/\varrho^2}\,dt \leq C+ C r^{-3} \varrho^3.
\end{align*}
Therefore, we obtain the estimate \eqref{eq13.03k} also in this case.
\end{parta}

\begin{partb}
By Lemma~\ref{lem2.19} and Part (a) above, we obtain the existence of the Green's matrix $\vec G(x,y)$ of $L$ in $\Omega$ and the estimate \eqref{eq13.03k} with $C=C(m,\nu,\mu_0, N_1)$.
By Theorem~\ref{thm2}, and observing that $2 R_{max}$ and $R_{max}$ are comparable to each other when $R_{max}<\infty$, we find that if $0<\abs{X-\tilde Y}_p<2 R_{max}$, then we have
\begin{equation}							\label{eq618ak}
\abs{\vec K(X,y)} \le C \left(1 \wedge \frac {d_x} {\abs{X-\tilde Y}_p} \right)^{\mu_0} \left(1 \wedge \frac{d_y} {\abs{X-\tilde Y}_p}\right)^{\mu_0} \abs{X-\tilde Y}_p^{-2}.
\end{equation}
We claim the following estimate holds:
For all $r \in (0, R_{max})$, we have
\begin{equation}							\label{eq12.39p}
\abs{\vec K(t,x,y)} \leq C \varrho r^{-3} \bigset{1\wedge (d_x/r)}^{\mu_0} \bigset{1\wedge (d_y/r)}^{\mu_0}e^{-2\nu (t-4r^2)/\varrho^2},\quad \forall t > 4r^2.
\end{equation}

Let us assume the claim \eqref{eq12.39p} for the moment and prove the estimate \eqref{eq12.01s}.
Similar to \eqref{eq6.21ee}, in the case when $0<\abs{x-y} \leq r:= (\varrho\wedge R_{max})/2$,  we estimate
\[
\abs{\vec G(x,y)} \leq \int_0^{\abs{x-y}^2} + \int_{\abs{x-y}^2}^{4r^2}+\int_{4 r^2}^\infty \abs{\vec K(t,x,y)}\,dt =: I_1+I_2+I_3,
\]
It follows from \eqref{eq618ak} that
\begin{align*}
I_1  &\leq C \left(1 \wedge \frac {d_x} {\abs{x-y}} \right)^{\mu_0} \left(1 \wedge \frac {d_y} {\abs{x-y}} \right)^{\mu_0},\\
I_2  & \leq C \left(1 \wedge \frac {d_x} {\abs{x-y}} \right)^{\mu_0} \left(1 \wedge \frac {d_y} {\abs{x-y}} \right)^{\mu_0} \left\{ 1+ \ln \left(\frac{r}{\abs{x-y}}\right)\right\}.
\end{align*}
Also, by \eqref{eq12.39p}, we obtain
\[
I_3 \le C \varrho^3 r^{-3} \bigset{1\wedge (d_x/r)}^{\mu_0} \bigset{1\wedge (d_y/r)}^{\mu_0}.
\]
Combining the above estimates for $I_1, I_2$, and $I_3$ together, we obtain \eqref{eq12.01s} in this case.

If $\abs{x-y} \geq r =(\varrho\wedge R_{max})/2$, then by using \eqref{eq618ak} and \eqref{eq12.39p}, we estimate
\begin{align*}
\abs{\vec G(x,y)} &\leq \int_0^{4r^2}+\int_{4r^2}^\infty \abs{\vec K(t,x,y)}\,dt\\
& \leq C \bigset{1\wedge (d_x/r)}^{\mu_0} \bigset{1\wedge (d_y/r)}^{\mu_0}+ C \varrho^3 r^{-3} \bigset{1\wedge (d_x/r)}^{\mu_0} \bigset{1\wedge (d_y/r)}^{\mu_0}\\
& \leq  C \varrho^3 r^{-3} \bigset{1\wedge (d_x/r)}^{\mu_0} \bigset{1\wedge (d_y/r)}^{\mu_0}.
\end{align*}
Therefore, we also obtain \eqref{eq12.01s} when $\abs{x-y} \geq (\varrho\wedge R_{max})/2$.

It only remains for us to prove \eqref{eq12.39p}. The strategy is similar to the proof of \eqref{eq3.8yy}.
We first prove the following estimate, which is a ``half'' of the estimate \eqref{eq12.39p}.
\begin{equation}							\label{eq2d.28}
\abs{\vec K(t,x,y)} \leq C \varrho r^{-3} \bigset{1\wedge (d_x/r)}^{\mu_0}e^{-2\nu (t-3r^2)/\varrho^2},\quad \forall t > 3r^2,\quad \forall r\in(0,R_{max}).
\end{equation}
To prove \eqref{eq2d.28}, it is enough to assume that $d_x< r/2$ because otherwise \eqref{eq2d.28} follows from \eqref{eq12.27g}. By Lemma~\ref{lem3.6} applied to $k$-th column of $\vec K(\cdot,y)$, for $k=1,\ldots, m$, we have
\[
\abs{\vec K(t,x,y)} \le C d_x^{\mu_0} r^{-2-\mu_0}\norm{\vec K(\cdot,y)}_{L^2(\cU_{r}^-(X))},\quad \forall t > 3r^2,\quad \forall r\in(0,R_{max}).
\]
Then by \eqref{eq12.27g}, we obtain
\[
\abs{\vec K(t,x,y)} \leq C \varrho d_x^{\mu_0}  r^{-3-\mu_0} e^{-2\nu (t-3r^2)/\varrho^2},\quad \forall t>3 r^2,\quad \forall r\in(0,R_{max}).
\]
Therefore, we proved the estimate \eqref{eq2d.28}.
Next, let us recall from \cite{DK09} that the Dirichlet heat kernel ${}^t\vec K(t,x,y)$ of the adjoint operator $\Lt$ is given by
\[
{}^t\vec K(t,x,y)=\vec K(t,y,x)^T.
\]
By repeating the above argument to ${}^t\vec K(\cdot,x)$, using the above identity, and utilizing the estimate \eqref{eq2d.28} instead of \eqref{eq12.27g}, we obtain \eqref{eq12.39p}.
This completes the proof of Part (b).
\end{partb}
The theorem is proved.
\hfill\qedsymbol

\subsection{Proof of Theorem~\ref{thm6.2g}}

We use the same notation used in the proof of Theorem~\ref{thm6.1g}.
Since $R_{max}=\infty$, we find that \eqref{eq618ak} valid for all $t>0$ and $x\in\Omega$.
For $x, y\in \Omega$ with $x\neq y$, let $r\geq\max (\abs{x-y},d_x\wedge d_y)$ be a number to be fixed later.
We write
\[
\int_0^\infty\abs{\vec K(t,x,y)}\,dt = \int_0^{\abs{x-y}^2} + \int_{\abs{x-y}^2}^{r^2}+ \int_{r^2}^\infty \abs{\vec K(t,x,y)}\,dt =: I_1+I_2+I_3.
\]
Then it follows from \eqref{eq618ak} that
\begin{align*}
I_1  &\leq C \left(1 \wedge \frac{d_x} {\abs{x-y}} \right)^{\mu_0} \left(1 \wedge \frac {d_y} {\abs{x-y}} \right)^{\mu_0},\\
I_2  & \leq C \left(1 \wedge \frac{d_x} {\abs{x-y}} \right)^{\mu_0} \left(1 \wedge \frac {d_y} {\abs{x-y}} \right)^{\mu_0}\cdot 2\ln \left(\frac{r}{\abs{x-y}}\right),\\
I_3 &\leq  C (d_x\wedge d_y)^{\mu_0} \int_{r^2}^\infty t^{-1-\mu_0/2}\,dt=  2C (d_x\wedge d_y)^{\mu_0} \mu_0^{-1}r^{-\mu_0},
\end{align*}
where $C=C(m,\nu, \mu_0, N_1)$.

We define $\vec G(x,y)$ by the formula \eqref{eq:g01}.
It is shown in the proof of \cite[Theorem~2.21]{DK09} that thus defined function $\vec G(x,y)$ is indeed the Green's matrix of $L$ in $\Omega$.

Now, we will choose $r$ as follows.
\begin{casei}
We take $r=d_x\wedge d_y$. Then
\[
\abs{\vec G(x,y)}\leq I_1+I_2+I_3 \leq C+ C \ln \left( d_x\wedge d_y /\abs{x-y}\right)+C,
\]
and thus, \eqref{eq12.01k} is satisfied in this case.
\end{casei}

\begin{caseii}
We take $r=\abs{x-y}$. Then we have
\[
\abs{\vec G(x,y)} \leq I_1+I_3 \leq C \left(1 \wedge \frac{d_x} {\abs{x-y}} \right)^{\mu_0} \left(1 \wedge \frac {d_y} {\abs{x-y}} \right)^{\mu_0}.
\]
It is clear that \eqref{eq12.01k} is satisfied also in this case.
\end{caseii}

\begin{caseiii}
We take $r=\abs{x-y}$, and obtain a better estimate for $I_3$.
\[
I_3 \leq  C d_x^{\mu_0} d_y^{\mu_0}\int_{r^2}^\infty t^{-1-\mu_0}\,dt=  C d_x^{\mu_0} d_y^{\mu_0} \mu_0^{-1}r^{-2\mu_0}.
\]
Therefore, we have
\[
\abs{\vec G(x,y)} \leq I_1+I_3 \leq C \left(1 \wedge \frac{d_x} {\abs{x-y}} \right)^{\mu_0} \left(1 \wedge \frac {d_y} {\abs{x-y}} \right)^{\mu_0},
\]
and thus \eqref{eq12.01k} follows in this case too.
\end{caseiii}

The proof is complete.
\hfill\qedsymbol


\section{Global estimates for systems with H\"older continuous coefficients}		 \label{sec:h}

In this section, we assume that the coefficients $A^{\alpha \beta}_{ij}$ in \eqref{eq0.0} belong to the class $\mathscr{C}^{\mu/2,\mu}$, where $\mu\in (0,1)$, and obtain improved global estimates for the Green's matrix in $\cU=\bR\times\Omega$, where $\Omega$ is a bounded $C^{1,\mu}$ domain.
We remark that these assumptions can be relaxed; see Remark~\ref{dini}.
Recall that we have the following gradient estimate; see e.g, \cite[Chapter 3]{Gi93}.

\begin{lemma}                                \label{lem7.1}
Under the assumptions above, there exist $R_{max}\in (0,\infty]$ and $C>0$ so that for all $X\in\cU=\bR\times\Omega$ and $0<R<R_{max}$, the following is true:
If $\vec u$ is a weak solution of $\cL \vec u=0$ in $\cU_R^-(X)$ vanishing on $\cS_R^-(X)$, then  we have
\begin{equation}				\label{eq12.20}
\abs{D \vec u}_{0;\cU_{R/2}^-(X)}  \leq C R^{-2-n/2} \norm{\vec u}_{L^2(\cU_R^-(X))}.
\end{equation}
Here, the constants $R_{max}$ and $C$ depend on $n, m, \nu$, $\vec A^{\alpha  \beta}$, and $\Omega$.
\end{lemma}

The following lemma is an immediate consequence of Lemma \ref{lem7.1}; c.f. Lemma~\ref{lem3.6}.
\begin{lemma}			\label{lem7.2}
Assume the conditions of Lemma \ref{lem7.1}.
For $R \in (0,R_{max})$ and $X \in\cU$ such that $d_X< R/2$,  let $\vec u$  be a weak solution of $\cL\vec u=0$ in $\cU_R^-(X)$ vanishing on $\cS_R^-(X)$.
Then, we have
\begin{equation*}			
\abs{\vec u(X)} \le C d_X R^{-2-n/2}\norm{\vec u}_{L^2(\cU_R^-(X))},
\end{equation*}
where the constant $C$ depends on $n, m, \nu$, $\vec A^{\alpha  \beta}$, and $\Omega$.
\end{lemma}

Now we state the main result of this section.

\begin{theorem}			\label{thm7.3}
Assume the conditions of Lemma \ref{lem7.1}.
Let $\vec \cG(X,Y)=\vec \cG(t,x,s,y)$ be the Green's matrix of $\cL$ in $\cU=\bR\times\Omega$ and let $\delta(X,Y)$ be defined as in \eqref{eq3.6mm}.
Then for all $X, Y \in \cU$ with $X\neq Y$, we have
\begin{equation}			\label{eq12.47}
\abs{\vec \cG(t,x,s,y)}\leq C\,\chi_{(0,\infty)}(t-s)\cdot \delta(X,Y)  \left\{(t-s)\wedge R_{max}^2\right\}^{-n/2}\exp\left(-\kappa \frac{\abs{x-y}^2}{t-s} \right),
\end{equation}
where $\kappa=\kappa(\nu)$ and $C$ is a constant depending on $n, m, \nu$, $\vec A^{\alpha  \beta}$, and $\Omega$.
\end{theorem}
\begin{proof}
We note that Lemma \ref{lem7.1} implies the condition $\LH$ with $\mu_0=1$, which in turn gives the existence of the Green's matrix $\vec \cG(X,Y)$ of $\cL$ in $\cU$.
Now we can repeat the argument in the proof of Theorem \ref{thm2} by using Lemma~\ref{lem7.2} instead of Lemma \ref{lem3.6}.
\end{proof}

\begin{remark} \label{dini}
Following \cite[Section 5]{Lieberman87}, the same gradient estimates \eqref{eq12.20} is available for systems with Dini continuous coefficients in Dini domains, which is more general than the H\"older conditions; see \cite{KN2} for the definition of Dini domains.
Since we only used the gradient estimate \eqref{eq12.20} for weak solutions in the proof of Theorem \ref{thm7.3}, the conclusion of the theorem is still available for systems with Dini continuous coefficients in Dini domains.
We would like to hereby thank YanYan Li for helpful discussion regarding the gradient estimates for systems with Dini continuous coefficients.
We remark that a weaker form of the estimate \eqref{eq12.47} for scalar Green's functions was obtained in \cite{Cho} by using the maximum principle argument.
\end{remark}

\section{Appendix}

\begin{lemma}				\label{lem2.19}
Let $\cU=\bR\times \Omega$ and assume the condition $\LH$.
Then, the condition $\LB$ is satisfied with the same $R_{max}$ and $N_0=N_0(n,m,\mu_0,N_1)$.
\end{lemma}
\begin{proof}
We shall only prove {\em i)} since the proof of {\em ii)} is very similar.
Let $\vec u$ be a weak solution of $\cL \vec u=0$ in $\cU_R^-(X)$ vanishing on $\cS_R^-(X)$, where $X\in \cU$ and $R\in (0,R_{max})$.
By using the triangle inequality, for all $Y\in Q_{R/2}^-(X)$ and $Z\in Q_{R/2}^-(Y)$, we have
\[
\abs{\tilde{\vec u}(Y)}^2 \leq 2\abs{\tilde{\vec u}(Y)-\tilde{\vec u}(Z)}^2+2\abs{\tilde{\vec u}(Z)}^2 \leq C R^{2\mu_0} [\tilde{\vec u}]_{\mu_0/2,\mu_0;Q_R^-(X)}^2+C\abs{\tilde{\vec u}(Z)}^2.
\]
Then by taking average over $Z\in Q_{R/2}^-(Y)$ and using $\LH$, we obtain
\[
\norm{\vec u}_{L^\infty(\cU_{R/2}^-(X))}^2 \le C R^{2\mu_0} [\tilde{\vec u}]_{\mu_0/2,\mu_0;Q_R^-(X)}^2+C R^{-n-2}\norm{\tilde{\vec u}}_{L^2(Q_R^-(X))}^2 \leq C R^{-n-2}\norm{\vec u}_{L^2(\cU_{R}^-(X))}^2,
\]
where $C=C(n,m,\mu_0, N_1)$.
The proof is complete.
\end{proof}

\begin{lemma}							\label{lem8.0a}
Assume that the operator $\cL$ is given as in \eqref{eq2.01aa} and satisfies the conditions \eqref{eqP-02} and \eqref{eqP-03}. Then the condition $\IH$ is equivalent to saying that the operator $\cL$ and its adjoint $\cLt$ satisfy the property $\PH$ in \cite{CDK}.
\end{lemma}
\begin{proof}
We shall only prove that part i) of $\IH$ is equivalent to the property $\PH$ for $\cL$ since the proof that part ii) of $\IH$ is equivalent to the property $\PH$ for $\cLt$ is very similar.

Let  $\vec u$ be a weak solution of $\cL \vec u=0$ in $Q_R^-(X)$, where $X\in \cU$ and $0<R<R_c\wedge d_X$.
Recall that the property $\PH$ in \cite{CDK} for $\cL$ is as follows: We have
\begin{equation}						\label{eq8.00dk}
\int_{Q_\rho^-(X)} \abs{D\vec u}^2\leq C_0\left(\frac{\rho}{r}\right)^{n+2\mu_0}\int_{Q_r^-(X)}\abs{D\vec u}^2, \quad \forall 0<\rho<r\leq R.
\end{equation}

First, we assume that the condition $\IH$ holds and prove \eqref{eq8.00dk}.
We may assume that $\rho<r/4$; otherwise, \eqref{eq8.00dk} is trivial.
We denote
\begin{equation*}						
(\vec u)_{Q_r^-}=\fint_{Q_r^-(X)} \vec u.
\end{equation*}
Notice that we may assume, by replacing $\vec u$ by $\vec u-(\vec u)_{Q_r^-}$, if necessary, that $(\vec u)_{Q_r^-}=0$.
From the energy inequality (see e.g., \cite[\S III.2]{LSU}), the condition $\IH$, and then parabolic Poincar\'e's inequality (see e.g., \cite[Lemma~2.4]{CDK}), it follows (recall that $\rho<r/4$)
\begin{align*}
\int_{Q_\rho^-(X)}\abs{D\vec u}^2 & \leq C \rho^{-2}\int_{Q_{2\rho}^-(X)} \abs{\vec u-(\vec u)_{Q_{2\rho}^-}}^2
 \leq C \rho^{-2}\int_{Q_{2\rho}^-(X)}\fint_{Q_{2\rho}^-(X)} \abs{\vec u(Y)-\vec u(Z)}^2\,dZ\,dY\\
&\leq C \rho^{n+2\mu_0}[\vec u]_{\mu_0/2,\mu_0; Q_{2\rho}^-(X)}^2 \leq C\rho^{n+2\mu_0} r^{-2\mu_0}\fint_{Q_r^-(X)}\abs{\vec u}^2\\
& \leq C\left(\frac{\rho}{r}\right)^{n+2\mu_0} \int_{Q_r^-(X)}\abs{D\vec u}^2,
\end{align*}
where $C=C(n,m,\nu,\mu_0,C_0)$. We have derived the property $\PH$ for $\cL$.

Next, we assume the property $\PH$ for $\cL$ and prove the condition $\IH$.
By the parabolic Poincar\'e's inequality (see e.g., \cite[Lemma~2.4]{CDK}) and the property $\PH$, and then the energy inequality, we obtain for all $Y\in Q_{R/4}^-(X)$ and $r\in (0,R/4]$ that
\begin{align*}
\int_{Q_r^-(Y)} \bigabs{\vec u-(\vec u)_{Q_r^-}}^2 &\leq C r^2\int_{Q_r^-(Y)}\abs{D\vec u}^2 \leq C r^2 \left(\frac{r}{R}\right)^{n+2\mu_0}  \int_{Q_{R/4}^-(Y)}\abs{D\vec u}^2\\
&\leq C\left(\frac{r}{R}\right)^{n+2+2\mu_0} \int_{Q_{R/2}^-(Y)} \abs{\vec u}^2 \leq C r^{n+2+2\mu_0} R^{-2\mu_0} \fint_{Q_R^-(X)} \abs{\vec u}^2,
\end{align*}
where $C=C(n,m,\nu,\mu_0,C_0)$. Then by the Campanato's characterization of H\"older continuous functions (see e.g., \cite[Lemma~2.5]{CDK}), we obtain
\[
[\vec u]_{\mu_0/2,\mu_0;Q_{R/4}^-(X)}^2 \leq C R^{-2\mu_0} \fint_{Q_R^-(X)} \abs{\vec u}^2.
\]
Then, the above inequality together with a standard covering argument yields part i) of the condition $\IH$.
The proof is complete.
\end{proof}

\begin{lemma}							\label{lem8.1}
Let $\cU=\bR\times \Omega$ and assume the condition $\LHP$ in Remark~\ref{rmk:P-03}.
Then the condition $\LH$ is satisfied.
\end{lemma}
\begin{proof}
We shall only demonstrate that part i) of $\LHP$ implies part i) of $\LH$ since the proof that part ii) of $\LHP$ implies part ii) of $\LH$ is very similar.

Let $\vec u$ be a weak solution of $\cL \vec u=0$ in $\cU_R^-(X)$ vanishing on $\cS_R^-(X)$, where $X\in \cU$ and $R\in (0,R_{max})$.
Notice that from $\LHP$ we have for all $Y\in Q_{R/4}^-(X)$ and $r\in (0,R/4]$ that
\begin{equation}						\label{eq8.01wt}
\int_{\cU_r^-(Y)} \abs{D \vec u}^2\leq C \left(\frac{r}{R}\right)^{n+2\mu_0}\int_{\cU_{R/4}^-(Y)}\abs{D\vec u}^2,
\end{equation}
where $C=4^{n+2\mu_0}C_0$.
By Lemma~\ref{lem8.1ap} below, the above estimate \eqref{eq8.01wt},  and the energy inequality, we have for all $Y\in Q_{R/4}^-(X)$ and $r\in (0,R/4]$ that
\begin{align*}
\int_{Q_r^-(Y)} \bigabs{\tilde{\vec u}-(\tilde{\vec u})_{Q_r^-}}^2 &\leq C r^2 \int_{\cU_r^-(Y)} \abs{D \vec u}^2\leq C r^2 \left(\frac{r}{R}\right)^{n+2\mu_0}\int_{\cU_{R/4}^-(Y)}\abs{D\vec u}^2\\
&\leq C\left(\frac{r}{R}\right)^{n+2+2\mu_0} \int_{\cU_{R/2}^-(Y)} \abs{\vec u}^2 \leq C r^{n+2+2\mu_0} R^{-2\mu_0} \fint_{Q_R^-(X)} \abs{\tilde{\vec u}}^2,
\end{align*}
where $C=C(n,m,\nu,\mu_0,C_0)$.
Then by the Campanato's characterization of H\"older continuous functions (see e.g., \cite[Lemma~2.5]{CDK}), we obtain
\[
[\tilde{\vec u}]_{\mu_0/2,\mu_0;Q_{R/4}^-(X)}^2 \leq C R^{-2\mu_0} \fint_{Q_R^-(X)} \abs{\tilde{\vec u}}^2.
\]
Then, the above inequality together with a standard covering argument yields part i) of the condition $\LH$.
The proof is complete.
\end{proof}

\begin{lemma}					\label{lem8.1ap}
Let $\cU=\bR\times\Omega$ and $\vec u$ be a weak solution of $\cL \vec u=\vec f$ in $\cU_r^-(X)$ vanishing on $\cS_r^-(X)$, where $\vec f \in L^\infty(\cU_r^-(X))$.
Then we have
\begin{equation}		\label{eq8.6ap}
\int_{Q_r^-(X)} \bigabs{\tilde{\vec u}-(\tilde{\vec u})_{Q_r^-}}^2 \leq C r^2\int_{\cU_r^-(X)}\abs{D\vec u}^2 +C r^{2-n}\norm{\vec f}_{L^1(\cU_r^-(X))}^2;\quad\tilde{\vec u}=\chi_{\cU_R^-(X)} \vec u,
\end{equation}
where $(\tilde{\vec u})_{Q_r^-}=\fint_{Q_r^-(X)} \tilde{\vec u}$ and $C=C(n,m,\nu)$.
\end{lemma}

\begin{proof}
We modify the proof of \cite[Lemma~3]{Struwe}.
Without loss of generality, we may assume $X=0$.
Let $\zeta=\zeta(x)$ be a smooth function defined on $\bR^n$ such that
\begin{equation*}				
0\leq \zeta \leq 1,\quad \supp \zeta \subset B_r,\quad \zeta\equiv 1\,\text{ on }\, B_{r/2},\quad\text{and}\quad  \abs{D\zeta} \le 4/r.
\end{equation*}
Let $\delta^{-1}=\int_{B_r}\zeta \geq  cr^n$, where $c=c(n)=2^{-n}\abs{B_1}$.
We denote
\[
\vec\beta(t):= \delta \int_{B_r} \zeta(x) \tilde{\vec u}(t,x)\,dx=\delta \int_{\Omega_r}\zeta(x) \vec u(t,x)\,dx;\quad \bar{\vec \beta}:=r^{-2}\int_{-r^2}^0 \vec \beta(t)\,dt.
\]
By following the proof of \cite[Lemma~4]{Struwe}, we estimate, for $-R^2 <s<t<0$,
\begin{equation}				\label{eq8.8ap}
\abs{\vec \beta(t)-\vec\beta(s)}^2 \leq C r^{-n} \int_{\cU^-_r} \abs{D\vec u}^2+ Cr^{-2n}\norm{\vec f}^2_{L^1(\cU_r^-)}.
\end{equation}
Since $(\tilde{\vec u})_{Q_r^-}$ minimizes the integral $\int_{Q_r^-} \abs{\tilde{\vec u}-\vec c}^2$ among $\vec c\in \bR^m$, we obtain
\[
\int_{Q_r^-} \bigabs{\tilde{\vec u}-(\tilde{\vec u})_{Q_r^-}}^2  \leq \int_{Q_r^-} \bigabs{\tilde{\vec u}-\bar{\vec \beta}}^2 \leq 2 \int_{Q_r^-} \bigabs{\tilde{\vec u}-\vec\beta(t)}^2+ 2 \int_{Q_r^-} \bigabs{\vec \beta(t)-\bar{\vec \beta}}^2.
\]

Notice that $\tilde{\vec u}\in W^{0,1}_2(Q_r^-)$ and $D \tilde{\vec u}= \chi_{\cU_r^-}\, D \vec u$ in $Q_r^-$.
Therefore, by a variant of Poincar\'e's inequality, we have
\[
 \int_{Q_r^-}\bigabs{\tilde{\vec u}-\vec\beta(t)}^2 =\int_{-r^2}^0 \int_{B_r} \bigabs{\tilde{\vec u}(t,x)-\vec\beta(t)}^2\,dx\,dt \leq C \int_{\cU_r^-} \abs{D\vec u}^2.
 \]
On the other hand, by the definition of $\bar{\vec \beta}$ and \eqref{eq8.8ap}, we obtain
\begin{align*}
\int_{Q_r^-} \bigabs{\vec \beta(t)-\bar{\vec \beta}}^2 & =  \abs{B_r} \int_{-r^2}^0 \bigabs{\vec \beta(t)-\bar{\vec \beta}}^2  \leq  Cr^{n-2} \int_{-r^2}^0 \int_{-r^2}^0 \bigabs{\vec\beta(t)-\vec\beta(s)}^2\,ds\,dt\\
&\leq C r^2 \int_{\cU^-_r} \abs{D\vec u}^2+ Cr^{2-n}\norm{\vec f}^2_{L^1(\cU_r^-)}.
\end{align*}
By combining the above inequalities, we obtain \eqref{eq8.6ap}. The proof is complete.
\end{proof}

\begin{lemma} \label{lem:G-06}
Assume $a^{\alpha\beta}(X)$ satisfy the condition \eqref{eqP-07}.
Let $\cU=\bR\times\Omega$, where $\Omega$ satisfies the condition $\CS$, and define $\mathscr{E}$ as in \eqref{eqP-08w}, where $A^{\alpha\beta}_{ij}(X)$ are the coefficients of the operator $\cL$.
There exists $\mathscr{E}_0= \mathscr{E}_0(n, \nu_0)>0$ such that if $\mathscr{E} <\mathscr{E}_0$, then the condition $\LH$ is satisfied with $\mu_0=\mu_0(n, \nu_0, \theta)$, $R_{max}=R_a$, and $N_1=N_1(n,m,\nu_0,\theta)$.
\end{lemma}

\begin{proof}
In this proof, we shall only consider part i) in the condition $\LH$ because proof of part ii) in the condition $\LH$ will be almost identical.

We shall prove below that there is  a positive number $\mathscr{E}_1=\mathscr{E}_1(n,\nu_0)$ such that if $\mathscr{E}<\mathscr{E}_1$, then the following holds:
There exist positive constants $\mu_1=\mu_1(n,\nu_0,\theta)$ and $C_1=C_1(n,m,\nu_0,\theta)$ such that for any $\tilde{X}\in \partial_x\cU=\bR\times\partial\Omega$ and $R\in (0,R_a)$, if $\vec u$ is a weak solution of $\cL \vec u=0$ in $\cU_R^-(\tilde X)$ vanishing on $\cS_R^-(\tilde X)$, then we have
\begin{equation}							\label{eqG-65}
\int_{\cU_\rho^-(\tilde{X})}\abs{D \vec u}^2 \le C_1 \left(\frac{\rho}{r}\right)^{n+2\mu_1} \int_{\cU_r^-(\tilde{X})}\abs{D \vec u}^2,\quad \forall 0<\rho<r \leq R.
\end{equation}
We remark that the above condition is a parabolic analogue of the property (BH) in \cite{HK07}.

We also note that by \cite[Lemma~2.2]{CDK}, there is $\mathscr{E}_2=\mathscr{E}_2(n,\nu_0)>0$ such that if $\mathscr{E}<\mathscr{E}_2$, then the property $\PH$ in \cite{CDK} is satisfied by $\cL$ with the exponent $\mu_2=\mu_2(n,\nu_0)$ and $R_c=\infty$.
More precisely, if $\vec u$ is a weak solution of $\cL \vec u=0$ in $Q_R^-(X)\subset \cU$, then we have
\begin{equation}							\label{eqG-65a}
\int_{Q_\rho^-(X)} \abs{D\vec u}^2\leq C_2\left(\frac{\rho}{r}\right)^{n+2\mu_2}\int_{Q_r^-(X)}\abs{D\vec u}^2, \quad \forall 0<\rho<r\leq R,
\end{equation}
where $C_2=C_2(n,m, \nu_0)$.
Then we combine \eqref{eqG-65} and \eqref{eqG-65a}, via a standard method in boundary regularity theory (see e.g., \cite[\S 3.4]{Gi93}) and Lemma~\ref{lem8.1}, to conclude that if $\mathscr{E}<\mathscr{E}_1 \wedge \mathscr{E}_2=:\mathscr{E}_0$, then the condition $\LH$ is satisfied with parameters $\mu_0=\mu_1\wedge \mu_2$, $N_1=N_1(n,m,\nu_0,\theta)$, and $R_{max}=R_a$; see Remark~\ref{rmk3.5}.

By the above observation, it only remains for us to prove the estimate \eqref{eqG-65}.
For $\tilde{X}\in \partial_p\cU$ and $R\in (0,R_a)$ given, let $\vec u$ be a weak solution of $\cL \vec u =0$ in $\cU_R^-(\tilde{X})$ vanishing on $\cS_R^-(\tilde{X})$.
Denote by $\cL_0$ the parabolic operator acting on scalar functions $v$ as follows:
\[
\cL_0 v = v_t - D_\alpha(a^{\alpha\beta}D_\beta v).
\]
For $r \in (0, R]$, let $v^i$ be a unique weak solution in $\cV_2(\cU_r^-(\tilde{X}))$ of the problem
\[
\left\{\begin{array}{rcl}
\cL_0 v^i=0&\text{in}&\cU_r^-(\tilde{X});\\
v^i= u^i&\text {on }&\partial_p \cU_r^-(\tilde{X}),
\end{array}\right.
\]
where $i=1,\ldots,m$.
We claim that there are positive constants $\mu=\mu(n,\nu_0,\theta)$ and $C=C(n,m,\nu_0,\theta)$ such that the following estimate holds:
\begin{equation}							\label{eqG-66}
\int_{\cU_\rho^-(\tilde{X})}\abs{D\vec v}^2 \leq C\left(\frac{\rho}{r}\right)^{n+2\mu}\int_{\cU_r^-(\tilde{X})}\abs{D\vec v}^2,\quad \forall 0<\rho<r.
\end{equation}

Notice that we may assume that $\rho<r/8$ because otherwise \eqref{eqG-66} becomes trivial.
Since each $v^i$ vanishes on $\cS_r^-(\tilde{X})$, it follows from \cite[Theorem~6.32]{Lieberman} and \cite[Theorem~6.30]{Lieberman} that there exist $\mu=\mu(n,\nu_0,\theta)>0$ and $C=C(n, \nu_0,\theta)>0$ such that
\begin{equation}							\label{eqG-67}
\osc_{\cU_{2\rho}^-(\tilde{X})} v^i \le C \rho^{\mu} r^{-\mu}\sup_{\cU_{r/4}^-(\tilde{X})} \bigabs{v^i} \leq C \rho^{\mu} r^{-\mu-n/2-1}\bignorm{v^i}_{L^2(\cU_{r/2}^-(\tilde{X}))}.
\end{equation}
In particular, the estimate \eqref{eqG-67} implies $v^i(\tilde{X})=0$.
Then, by the energy inequality and \cite[Lemma~4.2]{HK07},  we obtain (recall that $\rho<r/8$)
\begin{align*}
\int_{\cU_\rho^-(\tilde{X})}\bigabs{D v^i}^2 &\leq C \rho^{-2}\int_{\cU_{2\rho}^-(\tilde{X})}\bigabs{v^i}^2= C \rho^{-2}\int_{\cU_{2\rho}^-(\tilde{X})}\bigabs{v^i(Y)-v^i(\tilde{X})}^2\,dY \\
&\leq C \rho^{n}\left(\osc_{\cU_{2\rho}^-(\tilde{X})}\,v^i\right)^2 \leq C\left(\frac{\rho}{r}\right)^{n+2\mu} r^{-2}\int_{\cU_{r/2}^-(\tilde{X})}\bigabs{v^i}^2\\
&\le C\left(\frac{\rho}{r}\right)^{n+2\mu}\int_{\cU_{r}^-(\tilde{X})}\bigabs{D v^i}^2,\qquad i=1,\ldots,m.
\end{align*}
where $C=C(n,\nu_0,\theta)$.
This completes the proof of the estimate \eqref{eqG-66}.

Next, notice that $\vec w:=\vec u-\vec v$ belongs to $\cV_2(\cU_r^-(\tilde{X}))$, vanishes on $\partial_p \cU_r^-(\tilde{X})$, and satisfies
\[
\cL_0 \vec w = D_\alpha\left(\bigl(\vec A^{\alpha\beta}-a^{\alpha\beta}I_m \bigr)D_\beta\vec u\right).
\]
Therefore, by the energy inequality we obtain
\begin{equation}							\label{eqG-70w}
\int_{\cU_r^-(\tilde{X})}\abs{D\vec w}^2\leq C \mathscr{E}^2 \int_{\cU_{r}^-(\tilde{X})}\abs{D\vec u}^2,
\end{equation}
where $\mathscr{E}$ is defined as in \eqref{eqP-08w}.
By combining \eqref{eqG-66} and \eqref{eqG-70w}, we obtain
\[
\int_{\cU_\rho^-(\tilde{X})}\abs{D\vec u}^2 \leq C\left(\frac{\rho}{r}\right)^{n+2\mu} \int_{\cU_r^-(\tilde{X})}\abs{D\vec u}^2+C\mathscr{E}^2 \int_{\cU_r^-(\tilde{X})} \abs{D\vec u}^2, \quad \forall 0<\rho<r.
\]
Now, choose a number $\mu_1\in (0,\mu)$.
Then, by a well known iteration argument (see, e.g., \cite[\S III.2]{Gi83}), we find that there exists $\mathscr{E}_1$ such that if $\mathscr{E}< \mathscr{E}_1$, then we have the estimate \eqref{eqG-65}. The lemma is proved.
\end{proof}



\begin{thebibliography}{99}


\bibitem{Aronson}
Aronson, D. G.
\textit{Bounds for the fundamental solution of a parabolic equation}.
Bull. Amer. Math. Soc. \textbf{73} (1967), 890--896.

\bibitem{Auscher}
Auscher, P.
\textit{Regularity theorems and heat kernel for elliptic operators}.
J. London Math. Soc. (2) \textbf{54} (1996), no. 2, 284--296.

\bibitem{AMcT}
Auscher, P.; McIntosh, A.; Tchamitchian, Ph.
\textit{Heat kernels of second order complex elliptic operators and applications}.
J. Funct. Anal.  \textbf{152}  (1998),  no. 1, 22--73.

\bibitem{AQ}
Auscher, P.; Qafsaoui, M.
\textit{Equivalence between regularity theorems and heat kernel estimates for higher order elliptic operators and systems under divergence form}.
J. Funct. Anal.  \textbf{177}  (2000),  no. 2, 310--364.

\bibitem{AT}
Auscher, P.; Tchamitchian, Ph.
\textit{Square root problem for divergence operators and related topics}.
Ast\'{e}risque  No. 249 (1998)

\bibitem{AT2}
Auscher, P.; Tchamitchian, Ph.
\textit{Gaussian estimates for second order elliptic divergence operators on Lipschitz and $C\sp 1$ domains}.
Evolution equations and their applications in physical and life sciences (Bad Herrenalb, 1998),
15--32, Lecture Notes in Pure and Appl. Math., 215, Dekker, New York, 2001.



\bibitem{Cho}
Cho, S.
\textit{Two-sided global estimates of the Green's function of parabolic equations}.
Potential Anal. \textbf{25} (2006), no. 4, 387--398.


\bibitem{CDK}
Cho, S.; Dong, H.; Kim, S.
\textit{On the Green's matrices of strongly parabolic systems of second order}.
Indiana Univ. Math. J. \textbf{57} (2008), no. 4, 1633--1677.

\bibitem{Davies}
Davies, E. B.
\textit{Explicit constants for Gaussian upper bounds on heat kernels}.
Amer. J. Math. \textbf{109} (1987), no. 2, 319--333.

\bibitem{Davies89}
Davies, E. B.
\textit{Heat kernels and spectral theory}.
Cambridge Univ. Press, Cambridge, UK 1989.



\bibitem{DK09b}
Dong, H.; Kim, D.
\textit{On the $L_p$-solvability of higher order parabolic and elliptic systems with BMO coefficients}.
arXiv:1003.0969v1 [math.AP]

\bibitem{DK09}
Dong, H.; Kim, S.
\textit{Green's matrices of second order elliptic systems with measurable coefficients in two dimensional domains}.
Trans. Amer. Math. Soc. \textbf{361} (2009), no. 6, 3303-3323.

\bibitem{DM}
Dolzmann, G.; M\"{u}ller, S.
\textit{Estimates for Green's matrices of elliptic systems by $L\sp p$ theory}.
Manuscripta Math.  {\bf 88}  (1995),  no. 2, 261--273.

\bibitem{Eidelman}
Eidel'man, S. D.
\textit{Parabolic systems}.
Translated from the Russian by Scripta Technica, London North-Holland Publishing Co., Amsterdam-London; Wolters-Noordhoff Publishing, Groningen 1969.


\bibitem{FS}
Fabes, E. B.; Stroock, D. W.
\textit{A new proof of Moser's parabolic Harnack inequality using the old ideas of Nash}.
Arch. Rational Mech. Anal. \textbf{96} (1986), no. 4, 327--338.

\bibitem{Fuchs84}
Fuchs, M.
\textit{The Green-matrix for elliptic systems which satisfy the Legendre-Hadamard condition}.
Manuscripta Math. \textbf{46} (1984),  no. 1-3, 97--115.

\bibitem{Fuchs86}
Fuchs, M.
\textit{The Green matrix for strongly elliptic systems of second order
with continuous coefficients}.
Z. Anal. Anwendungen \textbf{5} (1986), no. 6, 507--531.

\bibitem{Gi83}
Giaquinta, M.
\textit{Multiple integrals in the calculus of variations and nonlinear elliptic systems}.
Princeton University Press, Princeton, NJ, 1983.

\bibitem{Gi93}
Giaquinta, M.
\textit{Introduction to regularity theory for nonlinear elliptic systems}.
Birkh\"auser Verlag, Basel, 1993.

\bibitem{GT}
Gilbarg, D.; Trudinger, N. S.
\textit{Elliptic partial differential equations of second order}.
Reprint of the 1998 ed.
Springer-Verlag, Berlin, 2001.

\bibitem{GW}
Gr\"uter, M.; Widman, K.-O.
\textit{The Green function for uniformly elliptic equations}.
Manuscripta Math. \textbf{37} (1982), no. 3, 303--342.

\bibitem{HK04}
Hofmann, S.; Kim, S.
\textit{Gaussian estimates for fundamental solutions to certain parabolic systems}.
Publ. Mat. \textbf{48} (2004), 481-496.

\bibitem{HK07}
Hofmann, S.; Kim, S.
\textit{The Green function estimates for strongly elliptic systems of second order}.
Manuscripta Math. \textbf{124} (2007), no. 2, 139-172.

\bibitem{KK09}
Kang, K.; Kim, S.
\textit{Global pointwise estimates for Green's matrix of second order elliptic systems}.
J. Differential Equations (2010) doi:10.1016/j.jde.2010.05.017;
arXiv:1001.2618v1 [math.AP]


\bibitem{Kim}
Kim, S.
\textit{Gaussian estimates for fundamental solutions of second order parabolic systems with time-independent coefficients}.
Trans. Amer. Math. Soc. \textbf{360} (2008), no. 11, 6031--6043.

\bibitem{Krylov} Krylov N. V.
\textit{Parabolic and elliptic equations with VMO coefficients}.
Comm. Partial Differential Equations, \textbf{32} (2007), no. 3, 453--475.

\bibitem{KN2}
Kukavica, I.; Nystr\"om, K.
\textit{Unique continuation on the boundary for Dini domains}.
Proc. Amer. Math. Soc. \textbf{126} (1998), no. 2, 441--446.

\bibitem{LSU}
Lady\v{z}enskaja, O. A.; Solonnikov, V. A.; Ural'ceva, N. N.
\textit{Linear and quasilinear equations of parabolic type}.
American Mathematical Society: Providence, RI, 1967.

\bibitem{Lieberman87}
Lieberman, G. M.
\textit{H\"older continuity of the gradient of solutions of uniformly parabolic equations with conormal boundary conditions},
Ann. Mat. Pura Appl. (4), \textbf{148} (1987), 77--99.

\bibitem{Lieberman}
Lieberman G. M.
\textit{Second order parabolic differential equations},
World Scientific Publishing Co., Inc., River Edge, NJ, 1996.


\bibitem{LSW}
Littman, W.; Stampacchia, G.; Weinberger, H. F.
\textit{Regular points for elliptic equations with discontinuous coefficients}.
Ann. Scuola Norm. Sup. Pisa (3) \textbf{17} (1963), 43--77.

\bibitem{MZ}
Mal\'y, J.; Ziemer, W. P.
\textit{Fine regularity of solutions of elliptic partial differential equations}.
American Mathematical Society, Providence, RI, 1997.


\bibitem{Moser}
Moser, J.
\textit{A Harnack inequality for parabolic differential equations}.
Comm. Pure Appl. Math. \textbf{17} (1964), 101--134.

\bibitem{Nash}
Nash, J.
\textit{Continuity of solutions of parabolic and elliptic equations}.
Amer. J. Math. \textbf{80} (1958), 931--954.

\bibitem{PE} Porper, F. O.; Eidel'man, S. D.
\textit{Two-sided estimates of the fundamental solutions of second-order parabolic equations and some applications of them}.
(Russian)  Uspekhi Mat. Nauk  \textbf{39} (1984), no. 3(237), 107--156;
English translation: Russian Math. Surveys \textbf{39} (1984), no. 3, 119--179.

\bibitem{Robinson}
Robinson, D. W.
\textit{Elliptic operators and Lie groups}.
Oxford Mathematical Monographs. Oxford Science Publications. The Clarendon Press, Oxford University Press, New York, 1991.
 



\bibitem{Struwe}
Struwe, M.
\textit{On the H\"older continuity of bounded weak solutions of quasilinear parabolic systems}.
Manuscripta Math. \textbf{35} (1981), no. 1-2, 125--145.

\bibitem{VSC}
Varopoulos, N. Th.; Saloff-Coste, L.; Coulhon, T.
\textit{Analysis and Geometry on Groups}.
Cambridge Tracts in Mathematics, \textbf{100}. Cambridge University Press, Cambridge, 1992.

\end{thebibliography}
\end{document}